\pdfoutput=1
\documentclass[12pt,amsmath,amssymb,longbibliography,nofootinbib,showkeys,eprint]{revtex4-1}
\usepackage[activate={true,nocompatibility},final,tracking=true,kerning=true,spacing=true]{microtype}
\usepackage{epigraph}
\usepackage{graphicx}
\usepackage{epstopdf}
\usepackage{braket}
\usepackage{bm}
\usepackage{numprint}
\usepackage{float}
\usepackage[T1,T2A]{fontenc}
\usepackage[utf8]{inputenc}
\usepackage{hyperref}
\usepackage{xr}
\usepackage{seqsplit}
\usepackage{amsthm}
\theoremstyle{definitions}
\newtheorem{theorem}{Theorem}
\newtheorem*{remark}{Remark}
\newtheorem*{note}{Note}

\begin{document}
\preprint{V.M.}
\title{On Lebesgue Integral Quadrature
}
\author{Vladislav Gennadievich \surname{Malyshkin}} 
\email{malyshki@ton.ioffe.ru}
\affiliation{Ioffe Institute, Politekhnicheskaya 26, St Petersburg, 194021, Russia}

\date{July, 7, 2018}

\begin{abstract}
\begin{verbatim}
$Id: LebesgueQuadratures.tex,v 1.160 2020/02/22 19:25:15 mal Exp $
\end{verbatim}
A new type of quadrature is developed.
The Gaussian quadrature, for a given measure,
finds optimal values of a function's argument (nodes)
and the corresponding weights.
In contrast, the Lebesgue quadrature developed in this paper,
finds optimal values of function (value--nodes) and the corresponding weights.
The Gaussian quadrature groups sums by function argument; it can be viewed as a $n$--point discrete measure, producing the Riemann integral.
The Lebesgue quadrature groups sums by function value; it can be viewed as a $n$--point discrete distribution, producing the Lebesgue integral.
Mathematically, the problem is reduced to a generalized eigenvalue problem:
Lebesgue quadrature value--nodes are the eigenvalues and
the corresponding weights are the square of the averaged eigenvectors.
A numerical estimation of an integral as the Lebesgue integral
is especially advantageous when analyzing irregular and stochastic processes.
The approach separates the outcome (value--nodes) and the probability of the outcome (weight).
For this reason, it is especially well--suited for the study of non--Gaussian processes.
The
\href{http://www.ioffe.ru/LNEPS/malyshkin/code_polynomials_quadratures.zip}{software}
implementing the theory is available from the authors.
\end{abstract}
\maketitle

\section{\label{intro}Introduction}
A Gaussian quadrature is typically considered as
``an integral calculation tool''.
However,
the quadrature itself can be considered
as a discrete measure\cite{totik}.
The major practical drawback of Gauss--type quadratures
is that they,  like a Riemann integral,
are finding the nodes in a function's argument space.
A very attractive idea is to
build a quadrature with the nodes in
a function's value space, a Lebesgue--type quadrature.
As with the Lebesgue integral,
such a quadrature can be applied to 
integration of irregular functions
and interpolating sampled measure by a discrete Lebesgue integral.
When implemented numerically
such an approach can give
a completely new look toward 
relaxation type processes analysis.
This is the goal of this paper.

\section{\label{Measure} Measure}
Consider a measure $d\mu$, a basis $Q_k(x)$,
and a function to integrate $f(x)$.
An example of the measure can be:
Chebyshev with $[-1:1]$ support  $d\mu=dx/\sqrt{1-x^2}$,
 Laguerre with $[0:\infty]$ support  $d\mu=dx\exp(-x)$,
experimental data sample $(f^{(l)},x^{(l)})$ of $l=1\dots M$ points (discrete $M$--point measure),
etc.
In this paper $Q_k(x)$ basis is a polynomial of the degree $k$, e.g. $x^k$
or some orthogonal polynomials basis,
the results are invariant with respect
to basis choice,  $Q_k(x)=x^k$ and $Q_k=T_k(x)$
give \textsl{identical} results, but numerical stability
can be drastically different\cite{beckermann1996numerical,2015arXiv151005510G}.
Introduce Paul Dirac quantum mechanic
\href{https://en.wikipedia.org/wiki/Bra%E2%80%93ket_notation}{bra--ket notation}
\cite{wiki:braketnotation} $\Bra{}$ and $\Ket{}$:
\begin{eqnarray}
  \Braket{Q_k f}&=&\int d\mu Q_k(x) f(t) \\
  \Braket{Q_j | f |Q_k }&=&\int d\mu Q_j(x) Q_k(x) f(t) \label{QfQmoments}
  \label{braket}
\end{eqnarray}
The problem we study in this paper is to estimate
a Lebesgue integral\cite{kolmogorovFA}
by an optimal $n$--point discrete measure (\ref{intLebesgue}).
\begin{align}
  \Braket{f}&=\int f d\mu
  \label{LebdMuKF}
\end{align}
We are going to apply the technique
originally
developed in Refs. \cite{2015arXiv151005510G,2016arXiv161107386V,ArxivMalyshkinMuse},
the main idea  is to consider not a traditional
interpolation of an observable $f$ as a linear superposition
of basis functions such as
\begin{align}
  \Braket{\left[f(x)-f_{LS}(x)\right]^2}&\rightarrow\min\label{norm2regrf}\\
  f_{LS}(x)&= \sum_{k=0}^{n-1} \beta_k Q_k(x) \label{linsuperF}  
\end{align}
but instead to introduce a wavefunction $\psi(x)$
as a linear superposition
of basis functions, then to average an observable $f(x)$
with the $\psi^2(x)d\mu$ weight:
\begin{eqnarray}
  \psi(x)&=&\sum_{j=0}^{n-1} \alpha_j Q_j(x) \label{psixdef} \\
  f_{\psi}&=& \frac{\Braket{\psi|f|\psi}}{\Braket{\psi|\psi}} 
  = \frac{\sum\limits_{j,k=0}^{n-1} \alpha_j\Braket{Q_j|f|Q_k}\alpha_j}
  {\sum\limits_{j,k=0}^{n-1} \alpha_j\Braket{Q_j|Q_k}\alpha_k}
  \label{fpsi}
\end{eqnarray}
With a positively defined matrix $\Braket{Q_j|Q_k}$ 
the generalized eigenvalue  problem:
\begin{align}
\sum\limits_{k=0}^{n-1} \Braket{Q_j|f|Q_k} \alpha^{[i]}_k &=
  \lambda^{[i]} \sum\limits_{k=0}^{n-1} \Braket{ Q_j|Q_k} \alpha^{[i]}_k
  \label{GEV} \\ 
 \psi^{[i]}(x)&=\sum\limits_{k=0}^{n-1} \alpha^{[i]}_k Q_k(x)
  \label{psiC}
\end{align}
has a unique solution.
Found eigenfunctions to be normalized as
$\Braket{\psi^{[i]}|\psi^{[j]}}=\delta_{ij}$. Then 
$\Braket{\psi^{[i]}|f|\psi^{[j]}}=\lambda^{[i]}\delta_{ij}$ ;
$\sum_{l,m=0}^{n-1} \alpha^{[i]}_l\Braket{Q_l|Q_m}\alpha^{[j]}_m=\delta_{ij}$ ; 
and 
$\lambda^{[i]}=\Braket{\left[\psi^{[i]}\right]^2f}\Big/\Braket{\left[\psi^{[i]}\right]^2}$. 

\subsection{\label{gaussQ}The Gaussian quadrature}
A $n$-point Gaussian quadrature $(x^{[i]},w^{[i]})$; $i=0\dots n-1$:
\begin{eqnarray}
  \int f(x) d\mu=\Braket{f}&\approx&\sum_{i=0}^{n-1} f(x^{[i]})w^{[i]}
  \label{intGauss}
\end{eqnarray}
on the measure $d\mu$ is integration formula (\ref{intGauss})
that is exact if $f(x)$ is a polynomial of a degree  $2n - 1$
or less, in other cases it can be considered
as an approximation
of the measure $d\mu$ by a discrete $n$--point measure $(x^{[i]},w^{[i]})$.
A question about an efficient numerical approach
to $(x^{[i]},w^{[i]})$ calculation is a subject of extensive work\cite{totik,gautschi2004orthogonal}.
In our recent work\cite{2015arXiv151005510G} we established,
that the most practical approach to obtain $(x^{[i]},w^{[i]})$
for an arbitrary measure (often available only through data sample) 
is to put $f=x$ in Eq. (\ref{GEV}) and to solve
the generalized eigenvalue  problem:
\begin{eqnarray}
\sum\limits_{k=0}^{n-1} \Braket{Q_j|x|Q_k} \alpha^{[i]}_k &=&
  \lambda^{[i]} \sum\limits_{k=0}^{n-1} \Braket{ Q_j|Q_k} \alpha^{[i]}_k
  \label{GEVgauss} \\
  x^{[i]}&=&\lambda^{[i]} \label{xi} \\
  w^{[i]}&=&\frac{1}{\left[\psi^{[i]}(x^{[i]})\right]^2}\label{wi}
\end{eqnarray}
The $n$--th order orthogonal polynomial
relatively the measure $d\mu$
is equal to the $\pi_n(x)=const\cdot(x-x^{[i]})\psi^{[i]}(x)=const\prod_{j=0}^{n-1}(x-x^{[j]})$.
The Gaussian quadrature nodes $x^{[i]}$ are (\ref{GEVgauss}) eigenvalues,
the weights are equal to inverse square of the eigenfunction at $x=x^{[i]}$
(the eigenfunctions are normalized as
$\Braket{\psi^{[i]}|\psi^{[i]}}=\sum_{j,k=0}^{n-1} \alpha^{[i]}_j\Braket{Q_j|Q_k}\alpha^{[i]}_k=1$).
The (\ref{GEVgauss}) is \textsl{exactly}
the threediagonal
Jacobi matrix eigenvalue problem (see Ref. \cite{killip2003sum} and references therein for a review),
but
written in the basis of $Q_k(x)$, not 
in the basis of $\pi_k(x)$
as typically studied.
Particularly, this makes it easy to  obtain
three term
recurrence coefficients $a_k$ and $b_k$
($x\pi_{k}=a_{k+1}\pi_{k+1}+b_{k}\pi_{k}+a_{k}\pi_{k-1}$)
from a sampled data
numerically: find the moments $\Braket{Q_m}$ ; $m=0\dots 2n-1$
and
obtain  orthogonal polynomials $\pi_k$ ; $k=0\dots n$ in $Q_k(x)$ basis;
then  calculate $a_k$ and $b_k$ using multiplication operator of $Q_k(x)$ basis functions,
see the method \texttt{\seqsplit{getAB()}} of
\href{http://www.ioffe.ru/LNEPS/malyshkin/code_polynomials_quadratures.zip}{provided software}.
An ability to use Chebyshev or Legendre basis as $Q_k(x)$
allows us to calculate the $a_k$ and $b_k$
to a very high order (hundreds).
The weight expression (\ref{wi})
is typically more convenient numerically than the one
with
the Christoffel function $K(x)$:
\begin{eqnarray}
  K(x)&=&\frac{1}{\sum_{j,k=0}^{n-1}Q_j(x)G^{-1}_{jk}Q_k(x)}
  =\frac{1}{\sum_{i=0}^{n-1} \left[\phi^{[i]}(x)\right]^2}
  \label{christoffelfun}
\end{eqnarray}
Here $G^{-1}_{jk}$ is  Gram matrix $G_{jk}=\Braket{Q_jQ_k}$ inverse;
in (\ref{christoffelfun}) the $\phi^{[i]}(x)$ is
an arbitrary orthogonal basis, such that $\Braket{\phi^{[i]}|\phi^{[j]}}=\delta_{ij}$, when $\phi^{[i]}(x)=\psi^{[i]}(x)$ obtain (\ref{wi}).

The Gaussian quadrature (\ref{intGauss}
can be considered as a Riemann integral formula,
its nodes $x^{[i]}$ select optimal positions of a function's argument,
they are $\|x\|$ operator eigenvalues (\ref{GEVgauss}),
this integration formula assumes that $f(x^{[i]})$ exist
and can be calculated.
As with any Riemann integral
it requires the $f(x)$ to be sufficiently regular
for an integral to exist.

\subsection{\label{LebesgueQuadrature}The Lebesgue quadrature}
The \href{https://en.wikipedia.org/wiki/Riemann_integral}{Riemann integral}
sums the measure of  all $[x:x+dx]$ intervals.
The \href{https://en.wikipedia.org/wiki/Lebesgue_integration#Intuitive_interpretation}{Lebesgue integral}
sums the measure of 
all $x$ intervals for which the value of function
is in the interval $[f:f+df]$,
see demonstrating Fig. 1 of Ref. \cite{ArxivMalyshkinMuse}.
Consider a $n$-point Lebesgue quadrature
$(f^{[i]},w^{[i]})$; $i=0\dots n-1$:
\begin{eqnarray}
  \int f(x) d\mu=\Braket{f}&\approx&\sum_{i=0}^{n-1} f^{[i]} w^{[i]}
  \label{intLebesgue}
\end{eqnarray}
Now quadrature nodes $f^{[i]}$ are in function \textsl{value} space,
not in function \textsl{argument} space as in (\ref{intGauss}).
We will call them the \textbf{value--nodes}.
To obtain the value--nodes and weights
of a Lebesgue quadrature for the measure $d\mu$ and function $f$
consider an arbitrary polynomial $P(x)$ of a degree $n-1$ or less
and expand it in (\ref{GEV}) eigenfunctions:
\begin{eqnarray}
  P(x)&=&\sum\limits_{i=0}^{n-1}\Braket{P|\psi^{[i]}}\psi^{[i]}(x)
  \label{constexpansion}
\end{eqnarray}
Taking into account that 
$\Braket{P|f|\psi^{[i]}}=\lambda^{[i]}\Braket{P|\psi^{[i]}}$
the expression for
$\Braket{P|f|S}$ can be written
(here $P(x)$ and $S(x)$ are arbitrary polynomials
of a degree $n-1$ or less):
\begin{eqnarray}
  \Braket{P|f|S}&=&\sum\limits_{i=0}^{n-1} \lambda^{[i]}
  \Braket{P|\psi^{[i]}}\Braket{S|\psi^{[i]}}
  \label{inegralGsumP2} \\
  \Braket{f}&=&\sum\limits_{i=0}^{n-1}\lambda^{[i]} \Braket{\psi^{[i]}}^2 
  \label{inegralGsum}
\end{eqnarray}
The (\ref{inegralGsum}) (the case $P=S=1$)
is eigenvalues averaged
with the weights  $\Braket{\psi^{[i]}}^2$
(note that $\Braket{\left[\psi^{[i]}\right]^2}=1$).
The (\ref{inegralGsum}) gives the Lebesgue quadrature value--nodes and weights:
\begin{eqnarray}
  f^{[i]}&=&\lambda^{[i]} \label{fiLeb} \\
  w^{[i]}&=& \Braket{\psi^{[i]}}^2 \label{wiLeb}
\end{eqnarray}
The Lebesgue quadrature 
can be considered as a Lebesgue integral
interpolating formula
by a $n$--point discrete  measure  (\ref{intLebesgue}).
The value--nodes $f^{[i]}$ select optimal positions of function values,
they are $\|f\|$ operator  eigenvalues (\ref{GEV}),
the weight $w^{[i]}$ is the measure corresponding
to the value $f^{[i]}$. The
weights (\ref{wiLeb}) give
\begin{align}
  \Braket{1}&=\sum_{i=0}^{n-1}w^{[i]}
  \label{Lebweightssum}
\end{align}
the same normalizing as for the Gaussian quadrature weights (\ref{wi}).
As with the Gaussian quadrature (\ref{intGauss})
the Lebesgue quadrature (\ref{intLebesgue}
is exact for some class of functions.

\begin{theorem}
  \label{theoremP}
If a $n$--point  Lebesgue quadrature (\ref{intLebesgue})
is constructed for a measure $d\mu$ and a function $f(x)$,
then any integral $\Braket{P(x)f(x)}$,
where $P(x)$ is a polynomial of a degree $2n-2$ or less,
can be evaluated from it exactly.
\end{theorem}
\begin{proof}
When $P(x)$ is of a degree $n-1$ or less, then apply (\ref{inegralGsumP2})
with $S=1$.
For a degree above $n-1$
expand $P(x)=\sum_{j,k=0}^{n-1}Q_j(x)M_{jk}Q_k(x)$.
The matrix $M_{jk}$ is non--unique,
but always exists and can be obtained
e.g. by synthetic division $P(x)=Q_{n-1}(x)q(x)+r(x)$,
or using density matrix approach of the Appendix \ref{densitymatrix}.
The integral $\Braket{fP(x)}=\sum_{j,k=0}^{n-1}\Braket{Q_j|f|Q_k}M_{jk}$
then can be evaluated using (\ref{inegralGsumP2}) formula:
\begin{align}
  \Braket{fP(x)}&=\sum_{i=0}^{n-1}\lambda^{[i]} w_{(P)}^{[i]}
=\sum_{i=0}^{n-1}\lambda^{[i]} \Braket{\psi^{[i]}|\widehat{P}|\psi^{[i]}}
  \label{fP} \\
  w_{(P)}^{[i]}&=\Braket{\psi^{[i]}|\widehat{P}|\psi^{[i]}}=
  \sum\limits_{j,k=0}^{n-1}\Braket{\psi^{[i]}|Q_j}
  M_{jk} \Braket{Q_k|\psi^{[i]}}
   \label{wP}
\end{align}
The formula (\ref{fP}) has the same
eigenvalues $\lambda^{[i]}$, but they are now averaged with the
weights $w_{(P)}^{[i]}$, that are not necessary positive as in (\ref{wiLeb}),
note that $\Braket{P(x)}=\sum_{i=0}^{n-1}w_{(P)}^{[i]}$.
\end{proof}
\begin{remark}
The Gaussian quadrature can be considered as a
special case of the Lebesgue quadrature. If one put $f=x$,
then $n$--point Lebesgue quadrature gives exact answer
for an integral $\Braket{fP(x)}$  with a polynomial $P(x)$ of a degree $2n-2$ or less,
is reduced to a quadrature that is exact
for a polynomial $xP(x)$ of a degree $2n-1$ or less,
i.e. to a Gaussian quadrature.
 When $f=x$
the Lebesgue quadrature value--nodes are equal to the Gaussian nodes.
The most remarkable feature of the Lebesgue quadrature
is that it directly estimates the distribution of $f$:
each $w^{[i]}$ from (\ref{wiLeb}) is the measure
of $f(x)\approx f^{[i]}$ sets. For an application of
this feature to the
optimal clustering problem see \cite{malyshkin2019radonnikodym}.
\end{remark}

Theorem \ref{theoremP}
gives an algorithm for $\Braket{fP(x)}$  integral calculation:
use the same value--nodes $f^{[i]}$ from (\ref{fiLeb}),
but the weights are now from (\ref{wP}).
The Lebesgue quadrature allows to obtain
the value of any $\Braket{fP(x)}$ integral,
adjusting only the weights, value--nodes remain the same,
what provides a range of opportunities in applications.

A question arises about the most convenient
way to store and apply a quadrature.
As both Gaussian and Lebesgue quadratures
are obtained from (\ref{GEV}) generalized eigenvalue 
problem, the  $n$ pairs $(\lambda^{[i]},\psi^{[i]})$
completely define the quadrature.
For the Gaussian quadrature (\ref{GEVgauss}) $f(x)=x$, the eigenvalues
are the nodes, the eigenvectors are
Lagrange interpolating polynomial
built on  $x^{[i]}$ roots
of orthogonal polynomial $\pi_n(x)$ degree $n$
relatively the measure $d\mu$:
$\psi^{[i]}(x)=const\cdot\pi_n(x)/(x-x^{[i]})$.
For (\ref{GEVgauss}) eigenvectors
$\Braket{x^n|\psi^{[i]}}=(x^{[i]})^n\Braket{\psi^{[i]}}$, the (\ref{wP})
is  then $w_{(P)}^{[i]}=P(x^{[i]})\Braket{\psi^{[i]}}^2$,
hence
it is more convenient to store a Gaussian quadrature
as $(x^{[i]},w^{[i]})$ pairs rather than as $(x^{[i]},\psi^{[i]})$ pairs.
For Lebesgue quadrature the $w_{(P)}^{[i]}$ dependence (\ref{wP}) on $P(x)$
is not that simple,
it
requires an access to eigenvectors $\psi^{[i]}$ to calculate,
for this reason it is more convenient to store a Lebesgue quadrature
as $(f^{[i]},\psi^{[i]})$ pairs
rather than as $(f^{[i]},w^{[i]})$ pairs.
The specific form of quadrature storage
is determined by application,
in any case
all the results are obtained from  defining the quadrature pairs
$(\lambda^{[i]},\psi^{[i]})$,
a unique solution of (\ref{GEV}) problem.
This uniqueness makes the basis $\psi^{[i]}(x)$ very attractive
for principal components expansion.
For example the variation (\ref{norm2regrf}) can be PCA expanded:
\begin{align}
  \Braket{\left[f(x)-f_{LS}(x)\right]^2}&
 =
  \Braket{f^2}
  -\sum\limits_{i=0}^{n-1}\left(f^{[i]}\right)^2w^{[i]}
=
  \Braket{\left(f-\overline{f}\right)^2}
  -\sum\limits_{i=0}^{n-1}\left(f^{[i]}-\overline{f}\right)^2w^{[i]} 
  \label{evexpansionStdev}
\end{align}
Here  $\overline{f}={\Braket{f}}/{\Braket{1}}$.
The difference between (\ref{evexpansionStdev})
and regular principal components is that the basis $\psi^{[i]}(x)$
of the Lebesgue quadrature is \textsl{unique}.
This removes the major limitation of a principal components method:
it's dependence
on the attributes scale.

\subsection{\label{RadonNikodim}Numerical Estimation Of Radon--Nikodym Derivative}
Radon--Nikodym derivative\cite{kolmogorovFA}
is typically considered
as a probability density $d\nu/d\mu$ relatively
two Lebesgue measures $d\nu$ and $d\mu$.
Consider $f=d\nu/d\mu$, then (\ref{GEV}) is generalized eigenvalue 
problem with $\Braket{Q_j|\frac{d\nu}{d\mu}|Q_j}$ and $\Braket{Q_j|Q_j}$
matrices (basis functions products $Q_jQ_k$ averaged with respect
to the measure
$d\nu$ and $d\mu$ respectively). If at least one of these
two matrices is positively
defined then (\ref{GEV}) has a unique solution.
\begin{theorem}
  \label{theoremRN}
  The eigenvalues $\lambda^{[i]}$
  $i=0\dots n-1$ are $d\nu/d\mu$ Radon--Nikodym derivative
  extremums in the basis of (\ref{GEV}).
\end{theorem}
  \begin{proof}
    Consider the first variation of 
    $\frac{\Braket{\psi|\frac{d\nu}{d\mu}|\psi}}{\Braket{\psi|\psi}}$
    in the state $\widetilde{\psi}(x)=\psi(x)+\delta\psi$, then
    \begin{eqnarray}
      \frac{\Braket{\psi+\delta\psi|\frac{d\nu}{d\mu}|\psi+\delta\psi}}{\Braket{\psi+\delta\psi|\psi+\delta\psi}}&=&
      \Braket{\psi|\frac{d\nu}{d\mu}|\psi} \nonumber \\
      &+&2\left[\Braket{\psi|\frac{d\nu}{d\mu}|\delta\psi}
      -\Braket{\psi|\frac{d\nu}{d\mu}|\psi}
      \Braket{\psi|\delta\psi}\right]+\dots
      \label{fvarpsi}
    \end{eqnarray}
    when $\Ket{\psi}$ is (\ref{GEV}) eigenvector,
    then the first variation (\ref{fvarpsi})
    (linear in $\delta\psi$) is zero
    because of $\Ket{\frac{d\nu}{d\mu}\Big|\psi}=\lambda\Ket{\psi}$
     relation for (\ref{GEV}) eigenvectors.
  \end{proof}
  \begin{remark}
    If $\delta\psi$ does not belong to the original basis space
    of (\ref{GEV})  problem ---
    then extremal property no longer holds.
  \end{remark}
  Other estimates of Radon--Nikodym
  derivative  can be easily expressed
  in terms of (\ref{GEV}) eigenvectors.
  For example Nevai operator \cite{nevai} is equal to
  eigenvalues $\lambda^{[i]}$
  averaged with the $\left[\psi^{[i]}(x)\right]^2$ weights:
  \begin{align}
    \frac{d\nu}{d\mu}(x)&=
    \frac{\sum\limits_{i=0}^{n-1}\lambda^{[i]}\left[\psi^{[i]}(x)\right]^2}
         {\sum\limits_{i=0}^{n-1}\left[\psi^{[i]}(x)\right]^2}
         \label{bnevai}
  \end{align}
  Other estimates,
   such as the ratio of
   two Christoffel functions\cite{BarrySimon}
   for the measures $d\nu$ and $d\mu$ if both are positive,
  can also be expressed in a form of $\lambda^{[i]}$ averaged,
  but with the other weights:
   \begin{align}
    \frac{d\nu}{d\mu}(x)&=
    \frac{\sum\limits_{i=0}^{n-1}\left(\lambda^{[i]}\right)^{\gamma}\left[\psi^{[i]}(x)\right]^2}
         {\sum\limits_{i=0}^{n-1}\left(\lambda^{[i]}\right)^{\gamma-1}\left[\psi^{[i]}(x)\right]^2}
         & &-1\le\gamma\le 1
         \label{bKd}
  \end{align}  
  Different estimators converge to each other for $n\to\infty$.
A weighted $\lambda^{[i]}$ type of expression
preserves the bounds:
  if original $f$ is $[f_L:f_H]$ bounded
then (\ref{bnevai}) is $[f_L:f_H]$ bounded as well;
this is an important difference from positive polynomials
interpolation\cite{bernard2009moments}, where only
a low bound (zero) is preserved.
  A distinguishing feature
  of Radon--Nikodym
  derivative estimate as (\ref{GEV}) spectrum is that
  it is not linked to the states localized in $x$--space (such as (\ref{bnevai})),
  but  instead
  is linked to extremal states of the Radon--Nikodym derivative $d\nu/d\mu$.
  
The $\psi^{[i]}(x)$ in (\ref{bnevai})
  is $\psi^{[i]}(x)=\sum_{k=0}^{n-1}\alpha^{[i]}_kQ_k(x)$, i.e. it can be considered
  as a distribution with a \textsl{single} support point $x$: the distribution moments are equal to $Q_k(x)$.
  Now assume $Q_k(x)$
  correspond to some \textsl{actual distribution} of $x$ and $q_k$ are
  the moments of this distribution. Then the $\frac{d\nu}{d\mu}(x)$ 
   is:
   \begin{eqnarray}
    \frac{d\nu}{d\mu}(x)&=&
    \frac{\sum\limits_{i=0}^{n-1}\lambda^{[i]}\left[\sum\limits_{k=0}^{n-1}\alpha_k^{[i]}q_k\right]^2}
         {\sum\limits_{i=0}^{n-1}\left[\sum\limits_{k=0}^{n-1}\alpha_k^{[i]}q_k\right]^2}
         \label{bqdist}
   \end{eqnarray}
   The (\ref{bqdist}) is averaged eigenvalues $\lambda^{[i]}$ with positive weights,
for $q_k=Q_k(x)$
 it coincides with $x$--localized (\ref{bnevai}).
 However the (\ref{bqdist}) is much more general, it allows to obtain
 a Radon--Nikodym derivative
   for non--localized states.
   The (\ref{bqdist}) is the value of the Radon--Nikodym derivative
   for a distribution with given $q_k$ moments.
   Such ``distributed'' states naturally arise, for example,
   in a distribution regression problem\cite{2015arXiv151107085G,2015arXiv151109058G},
   where a bag of $x$--observations is mapped to a single $f$--observation.   
   There is one more generalization, considered in\cite{malyshkin2015norm,ArxivMalyshkinMuse}:
   density matrix mixed states, that cannot be reduced to a  pure state of a $\psi(x)$ form,
   we are going to discuss this generalization elsewhere,
   for a few simple examples
   see Appendix \ref{densitymatrix}, where a density matrix corresponding
   to a given polynomial is constructed and Appendix \ref{christoffelSpectrum},
  where a density matrix corresponding to the Chrisoffel function (\ref{christoffelfun})
  is constructed.
  Our approach 
  can estimate both: the measure (as a Lebesgue quadrature)
 and two measures density (as a Radon--Nikodym
 derivative), together with
\href{http://www.ioffe.ru/LNEPS/malyshkin/code_polynomials_quadratures.zip}{provided}
numerical implementation,
  this makes the approach extremely attractive to a number of practical problems,
  for example to joint probability estimation\cite{ArxivMalyshkinJointDistribution}.
  
\section{\label{numerical}Numerical Estimation}
The $(\lambda^{[i]},\psi^{[i]})$ pairs of (\ref{GEV}) eigenproblem
(for a Gaussian quadrature with $\Braket{Q_j | x |Q_k }$ and $\Braket{Q_j |Q_k }$ matrices,
and for a Lebesgue one with $\Braket{Q_j | f |Q_k }$ and $\Braket{Q_j |Q_k }$ matrices)
are required to calculate a quadrature.
A question arise about numerically most stable and efficient way of doing
the calculations.
Any $\Braket{Q_j | f |Q_k }$ matrix ($j,k=0\dots n-1$)
can  be calculated from the $\Braket{Q_m f}$ moments ($m=0\dots 2n-2$)
 using multiplication operator:
\begin{eqnarray}
  Q_j Q_k&=&\sum_{m=0}^{j+k}c_m^{jk}Q_m \label{cmul}
\end{eqnarray}
The value of $c_m^{jk}$ is analytically known
(see numerical implementation in the Appendix A of Ref. \cite{2015arXiv151005510G})
for four numerically  stable $Q_k(x)$ bases:
Chebyshev, Legendre, Hermite, Laguerre,
and for a basis with given three term
recurrence coefficients $a_k$ and $b_k$
it can be calculated numerically\footnote{
See the class \texttt{\seqsplit{com/polytechnik/utils/RecurrenceAB.java}}
of
\href{http://www.ioffe.ru/LNEPS/malyshkin/code_polynomials_quadratures.zip}{provided software}.
}
(all the bases give mathematically
identical results, because
(\ref{GEV}) is invariant with respect to an arbitrary
non--degenerated linear transform of the basis,
but numerical stability of the calculations depends greatly on basis choice).

Once the matrices $\Braket{Q_j | f |Q_k }$ and $\Braket{Q_j |Q_k }$
are calculated the (\ref{GEV}) can be solved using e.g. 
\href{http://www.netlib.org/lapack/lug/node54.html}{generalized eigenvalue problem}
subroutines from Lapack\cite{lapack}. With a good basis choice
numerically stable results can be obtained
for a 2D problem\cite{2015arXiv151101887G}
with up to $100\times 100$ elements in basis, i.e. for $10,000$ basis functions.

In Appendix A \& B of Ref. \cite{2015arXiv151005510G}
the description of
API and \verb+java+ implementation of polynomial operations in 
Chebyshev, Legendre, HermiteE, Laguerre, Shifted Legendre, Monomials  bases
is presented. The code is available from\cite{polynomialcode},
file \href{http://www.ioffe.ru/LNEPS/malyshkin/code_polynomials_quadratures.zip}{\texttt{\seqsplit{code\_polynomials\_quadratures.zip}}}.
 See 
 the program \texttt{\seqsplit{com/polytechnik/algorithms/ExampleRadonNikodym\_F\_and\_DF.java}}
for usage example.
This program reads $(x^{(l)},f^{(l)})$ pairs from a tab--separated file,
then calculates (\ref{fiLeb}) value--nodes and (\ref{wiLeb}) weights
for Lebesgue integral of the functions: $f(x)$, $df/dx$ with the measure $d\mu=dx$,
and $\frac{1}{f}df/dx$ with the measure $d\mu=fdx$,
see Ref. \cite{2016arXiv161107386V} for a description,
and Appendix \ref{RNFDFusage} for an example.
As a proof--of--concept a simple
matlab/\href{https://www.gnu.org/software/octave/}{octave}
implementation \texttt{\seqsplit{com/polytechnik/utils/LebesgueQuadratureWithEVData.m}}
is also provided,
the class calculates
the Lebesgue quadrature value--nodes and weights $(f^{[i]},w^{[i]})$
either from two matrices, or, second option, given $f(x)$
in an analytic form, calculates two matrices first
and then finds
the Lebesgue quadrature.
Usage demonstration in available from
\texttt{\seqsplit{com/polytechnik/utils/LebesgueQuadratures\_selftest.m}}.
This unoptimized
code calculates $\Braket{Q_j | f |Q_k }$ and $\Braket{Q_j |Q_k }$ matrices
in monomials and Chebyshev bases, then builds Gaussian and Lebesgue quadratures.

\section{\label{conclusion}Conclusion}
Obtained Lebesgue quadrature
is a new class of quadratures, besides being
suitable for
$\Braket{fP(x)}$ integrals estimation,
it can be applied to an estimation of the distribution of $f$:
each $w^{[i]}$ from (\ref{wiLeb}) is the measure
of $f(x)\approx f^{[i]}$ sets.
This is especially important for $f(x)$ of relaxation type,
this approach is superior to typically used approaches
based on 
$\Braket{f}$, $\Braket{f^2}$, $\Braket{f^3}$, $\Braket{f^3}$,
skewness and kurtosis approaches\cite{malyshkin2018spikes}.
In our early works\cite{2016arXiv161107386V,liionizversiyaran}
the (\ref{GEV}) equation was obtained, but
all the eigenvalues were considered to have equal weights,
their distribution
 was interpreted as a one related to
the distribution of $f(x)$,
this is similar to an interpretation of eigenvalues distribution
used in random matrix theory\cite{guhr1998random}.

In this paper an important step forward is made.
An eigenvalue $\lambda^{[i]}$ should have
the Lebesgue quadratures weight (\ref{wiLeb})
$\Braket{\psi^{[i]}}^2$, not the same weight as in our previous works
(first time the Eq. (\ref{wiLeb}) was obtained in Ref. \cite{malyshkin2015norm}
as cluster coverage,  formula (20) for $C^{[i]}$,
but it's importance was not then understood).

\begin{figure}[t]
  \includegraphics[width=5cm]{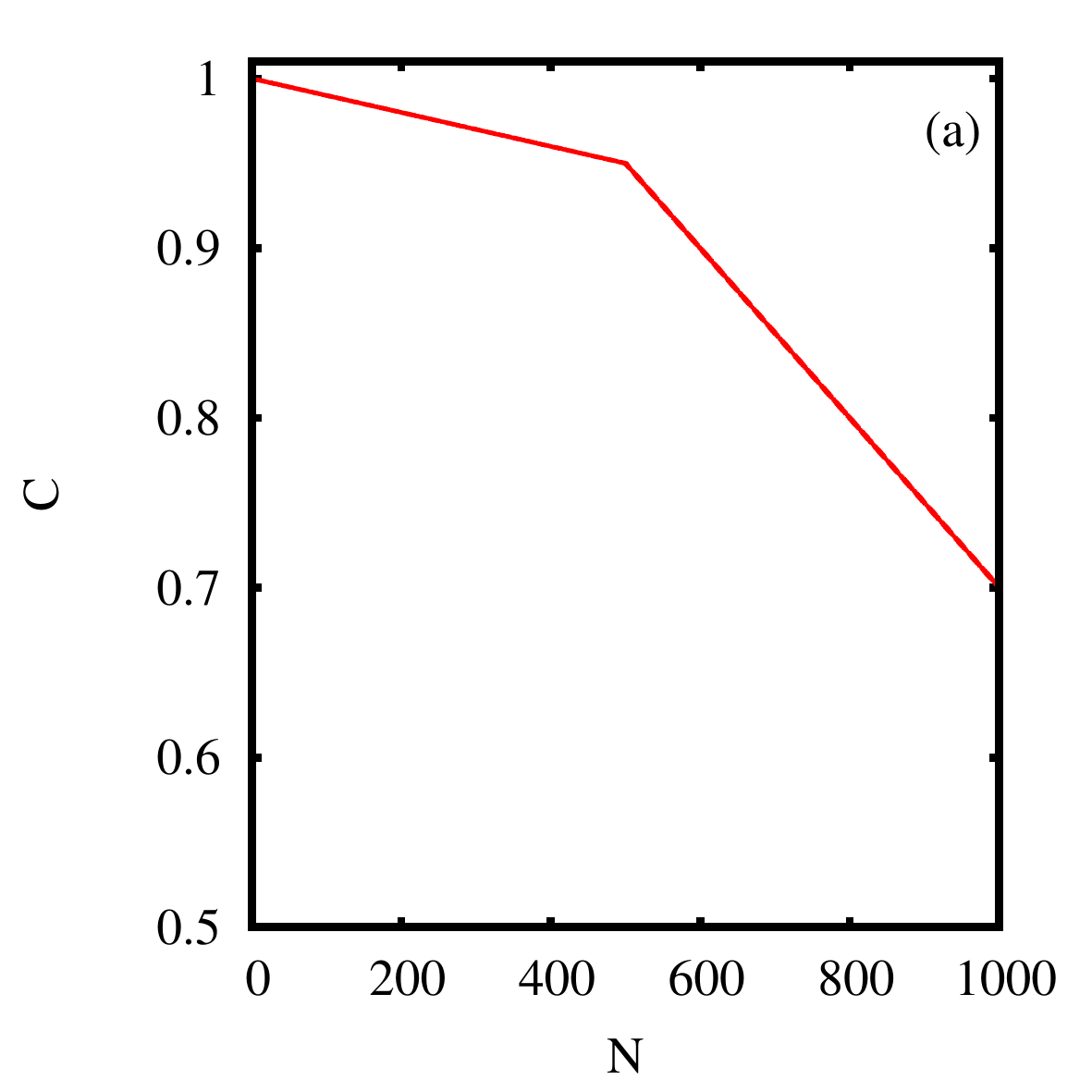}
\includegraphics[width=5cm]{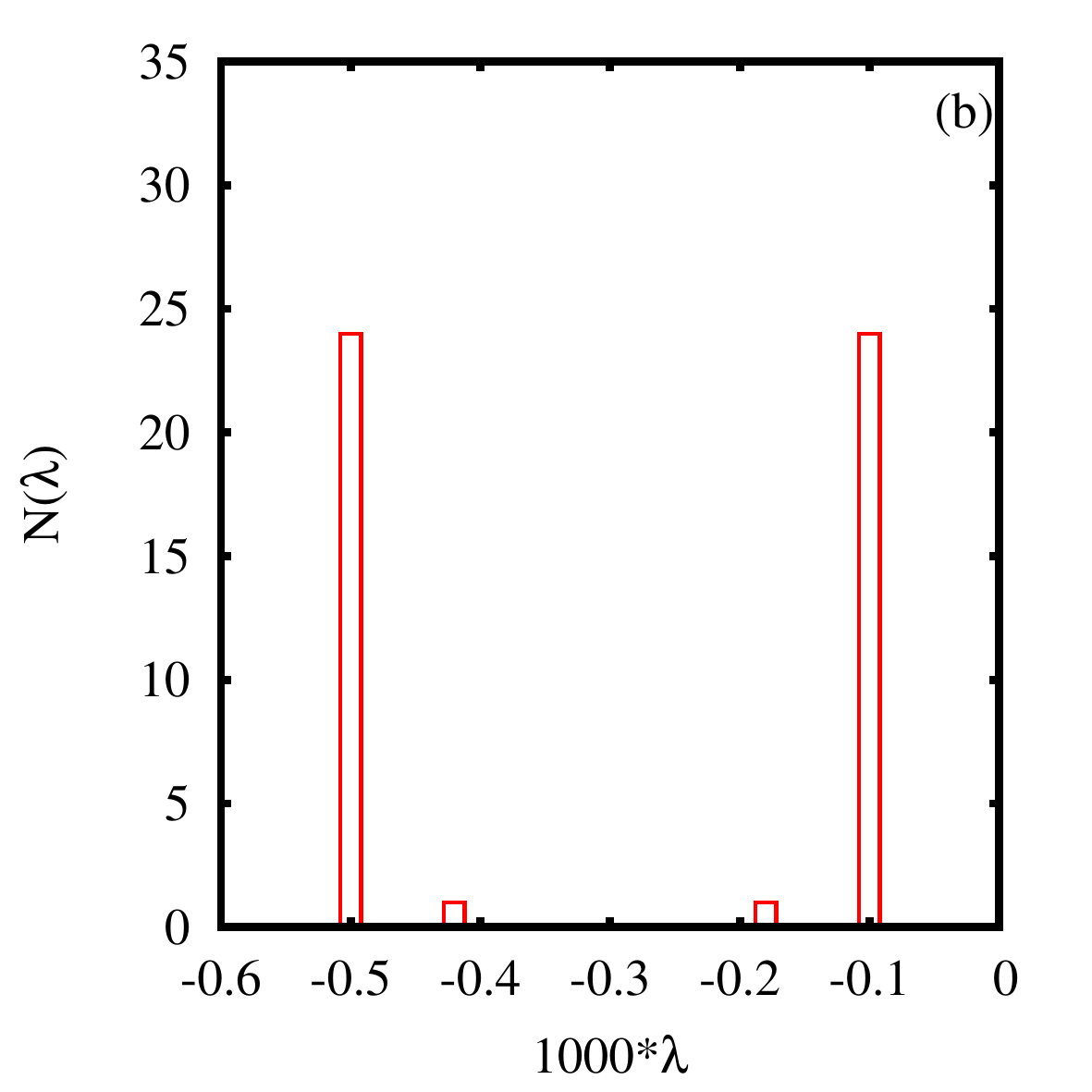}
\includegraphics[width=5cm]{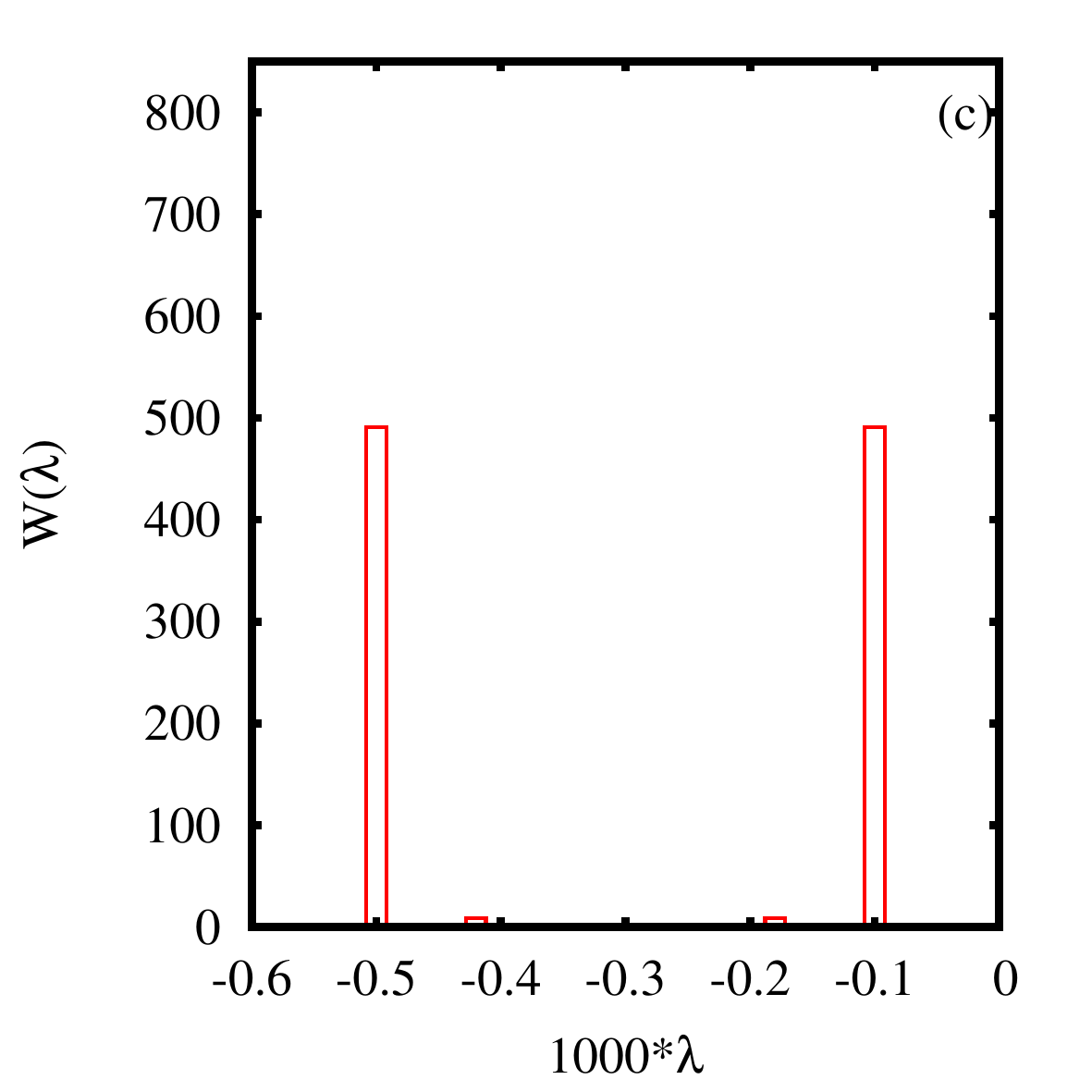} \\
\includegraphics[width=5cm]{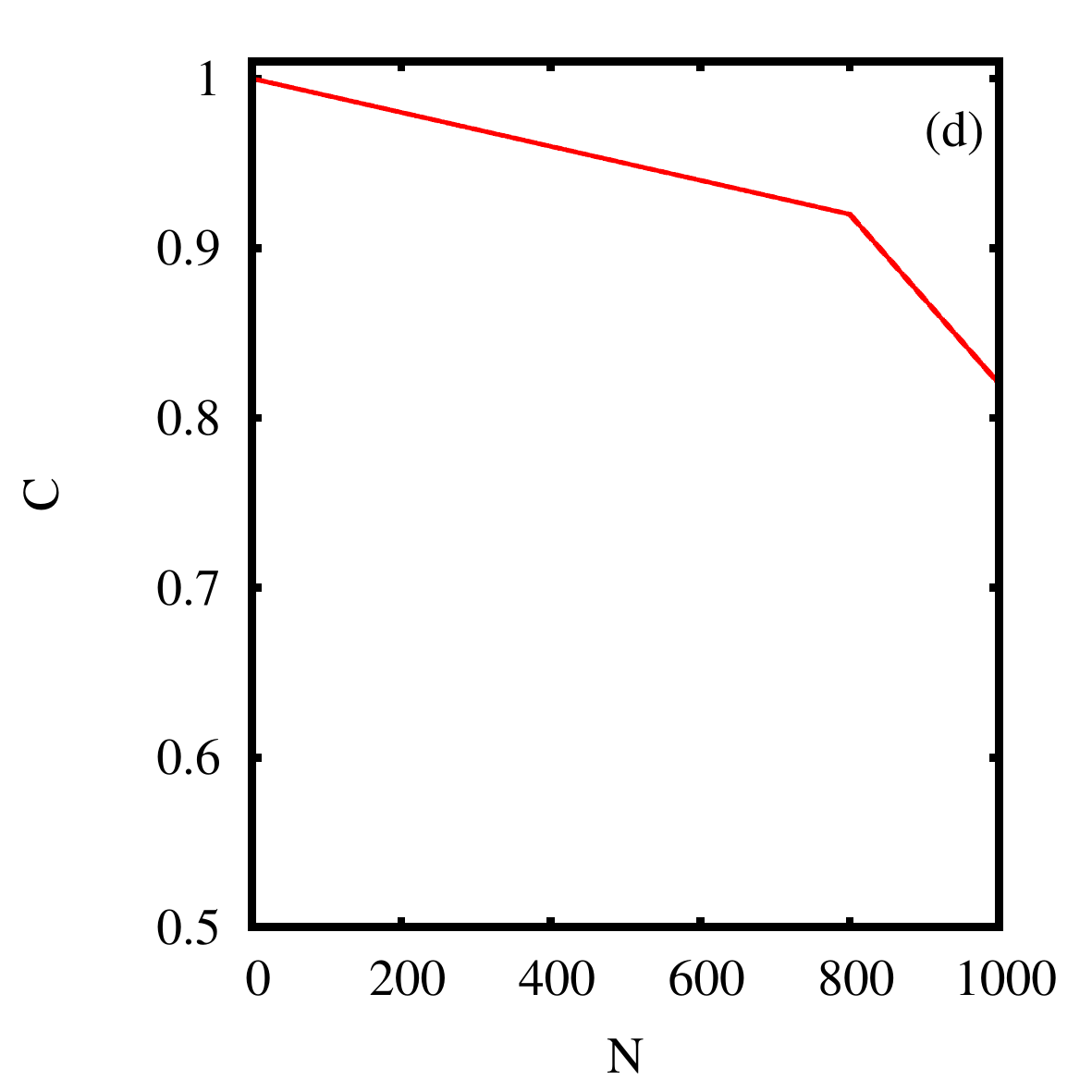}
\includegraphics[width=5cm]{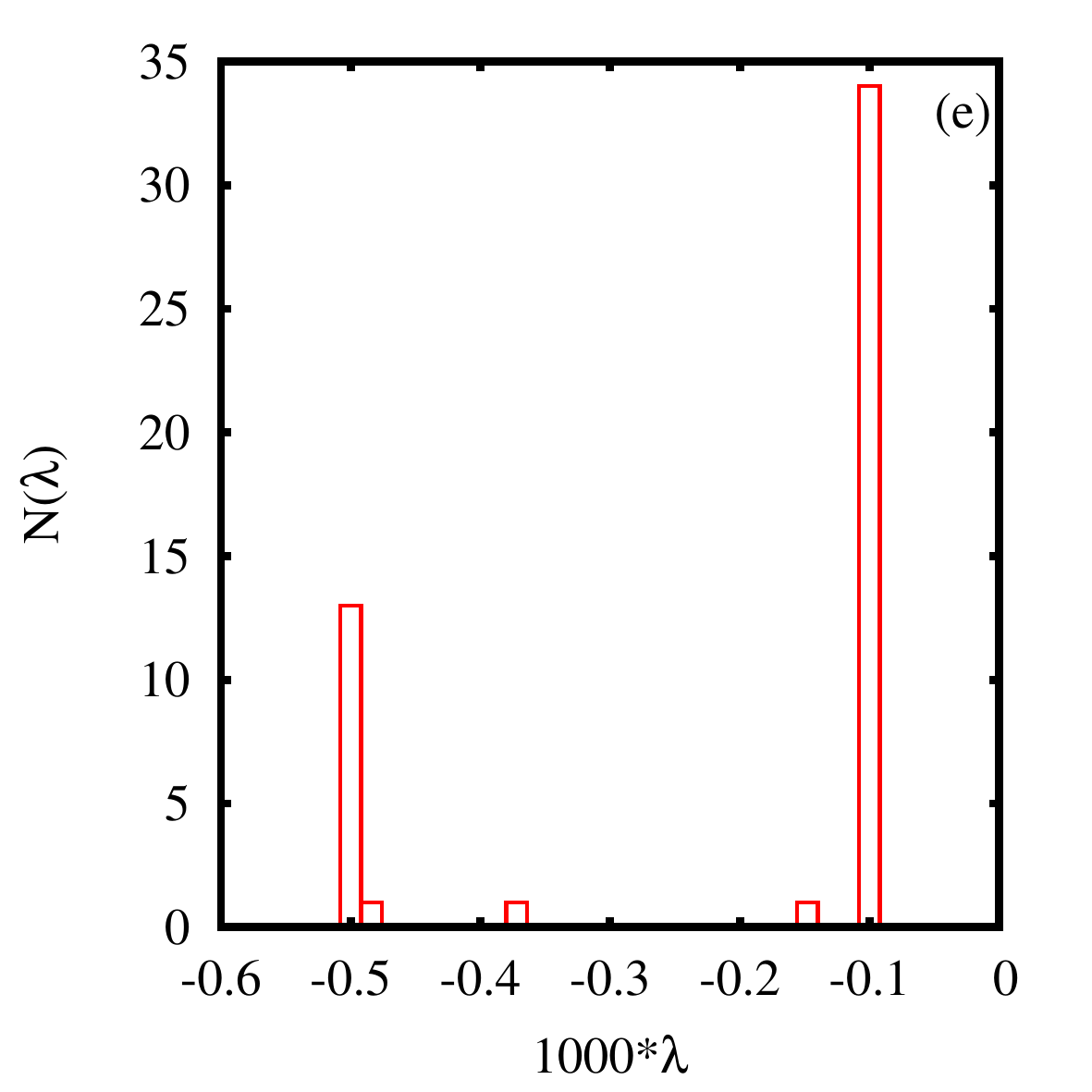}
\includegraphics[width=5cm]{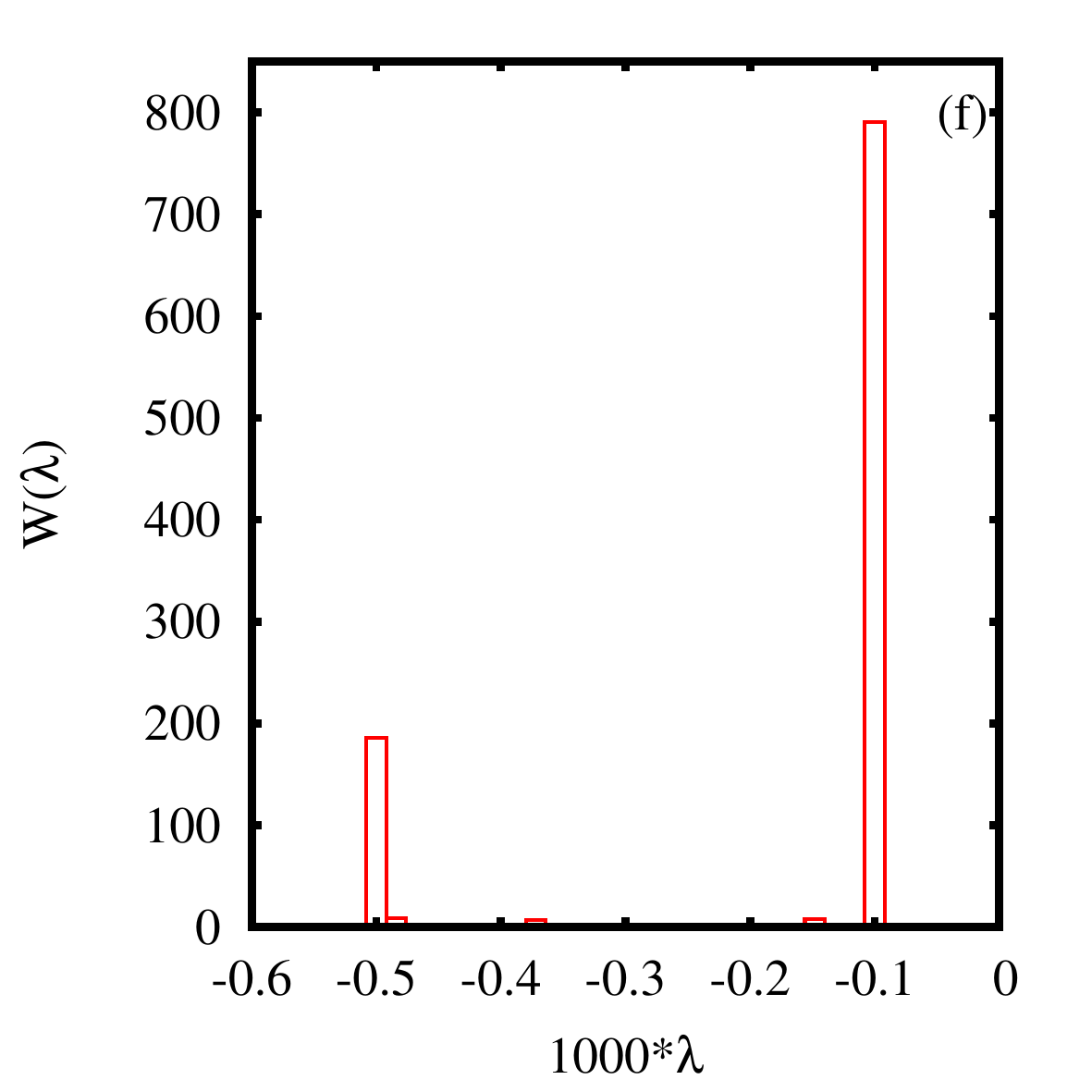}

  \caption{\label{figLiIon}
  Two stage degradation model with
  the slope on first and second stages
  $-10^{-4}$ and $-5\cdot 10^{-4}$ per cycle respectively.
  The stages length is 500:500 for (a), (b), (c)
  and 800:200 for (d), (e), (f).
  The (a) and (d) are $C(N)$ models for which
  $f=dC/dN$ is put to (\ref{GEV}).
  The (b) and (e) are the distributions of $\lambda^{[i]}$ from (\ref{GEV}) with equal weights,
  Ref. \cite{liionizversiyaran} results.
  The (c) and (f) are the distributions of $\lambda^{[i]}$
  with (\ref{wiLeb}) weights, the
   peak height corresponds exactly to the stage length because of chosen measure $d\mu=dN$.
   The calculations are performed for  $n=50$ in polynomial basis. 
  }
\end{figure}

To demonstrate the difference
in weights accounting
 take two--stage degradation
data model from Ref. \cite{liionizversiyaran}.
Li--ion batteries capacity fade with each cycle,
the degradation rate per cycle $dC/dN$
is the characteristics of interest.
Consider $x=N$ and the measure $d\mu=dN$ (recent and old cycles are equally important),
use $f(x)$ as battery degradation rate $f=dC/dN$.
As in Ref. \cite{liionizversiyaran} consider $C(N)$
for 1000 cycles, the degradation rate
for the first and second stages
is $10^{-4}$ and $5\cdot 10^{-4}$ per cycle respectively.
Two processes with first:second stages ratio
as 500:500 ($f=-10^{-4}$ for $0\leq x\leq 500$ ;
$f=-5\cdot 10^{-4}$ for $500\leq x\leq 1000$) and
800:200 ($f=-10^{-4}$ for $0\leq x\leq 800$ ;
$f=-5\cdot 10^{-4}$ for $800\leq x\leq 1000$)
are used as the model data, Fig. \ref{figLiIon}.
In our previous works\cite{2016arXiv161107386V,liionizversiyaran}
we established, that the distribution of $\lambda^{[i]}$
from (\ref{GEV})
is related to the distribution of $f$. In this paper
this relation is found, the weights are (\ref{wiLeb})
Lebesgue quadrature weights. Note, that for the data in Fig. \ref{figLiIon},
the peaks height for (c) and (f) 
correspond exactly to stage length, because
of the measure chosen $d\mu=dN$.

A Lebesgue quadrature $(f^{[i]},w^{[i]})$ 
can be interpreted as $f(x)$ discrete distribution.
The selection of value--nodes is optimal,
such a quadrature performs optimal $n$--point discretization of $f(x)$.
The approach is applicable to non--Gaussian distributions
(e.g. with infinite standard deviation (but \textsl{not} with infinite mean),
burst of many orders of magnitude, etc.).
The situation is similar to the one in quantum mechanics:
when a quantum Hamiltonian is known incorrectly
and have some energy state, that is greatly different
from the ground state,
such a state does not change system behavior at all,
because it has close to zero probability.
The Lebesgue quadrature has similar ideology, it
separates the state on: an observable value $f^{[i]}$
and the probability of it $w^{[i]}$.
Similar path have been successfully tried earlier in our
quantum--mechanics approach to machine learning of 
Ref. \cite{malyshkin2015norm},
where we separated
system properties (described by the outcomes)
and system testing conditions (described by the coverage).

\section*{Acknowledgment}
  Vladislav Malyshkin would like to thank
  \href{https://www.researchgate.net/profile/S_Bozhokin}{S. V. Bozhokin}
  and \href{https://iamm.spbstu.ru/person/komarchev_ivan_anatolevich/}{I. A. Komarchev}
  from \href{http://www.spbstu.ru/}{Peter the Great St.Petersburg Polytechnic University}
  for fruitful discussions.

\appendix
\section{\label{densitymatrix}Density matrix,
corresponding to a given polynomial}
In Section \ref{LebesgueQuadrature}
the integral $\Braket{P(x)f(x)}$ with a polynomial $P(x)$ of a degree
$2n-2$ or less is considered.
The technique of \cite{2015arXiv151005510G} deals mostly
with $\Braket{\psi^2(x)f(x)}=\Braket{\psi|f|\psi}$ type of integrals,
and it is of practical value to be able to reduce a state described
by an arbitrary polynomial:
\begin{align}
  P(x)&=\sum_{k=0}^{2n-2}\gamma_kQ_k(x)
\end{align}
to the state described by the density matrix:
\begin{align}
  \rho(x,y)&=\sum_{i=0}^{n-1}\lambda^{[i]} \psi^{[i]}(x)  \psi^{[i]}(y)  \label{rho}\\
  P(x)&= \rho(x,x) \label{rhoPrelation}
\end{align}  
such that $P(x)=\rho(x,x)$, and $\lambda^{[i]} ; \psi^{[i]}(x)$
are the eigenvalues and the eigenvectors of some operator $\|\rho\|$.
\begin{theorem}
  \label{theoremPoper}
For a non--degenerated basis $Q_k(x)$ relatively the measure $d\mu$
such operator always exists and is generated by a measure
with the moments $\Braket{Q_k(x)}_P$.
\end{theorem}
\begin{proof}
  To find a measure, such that
  $P(x)=\sum_{j,s,t,k=0}^{n-1}Q_j(x)\left[G^{-1}_{js}\Braket{Q_sQ_t}_P G^{-1}_{tk}\right] Q_k(x)$
  (here $G^{-1}_{jk}$ is  Gram matrix $G_{jk}=\Braket{Q_jQ_k}$ inverse)
  apply
  multiplication operator $c_m^{jk}$ from (\ref{cmul})
  to obtain:
  \begin{align}
    \sum\limits_{m=0}^{2n-2}\gamma_mQ_m(x)&=
    \sum\limits_{j,s,t,k=0}^{n-1}\,\sum\limits_{m=0}^{j+k}\,\sum\limits_{l=0}^{s+t} c_m^{jk} G^{-1}_{js} c_l^{st} G^{-1}_{tk} \Braket{Q_l}_P Q_m(x)
  \end{align}
  Comparing the coefficients by $Q_m(x)$ obtain a
  linear system of $2n-1$ dimension,
  from which the $\Braket{Q_l}_P$ ; $l=0\dots 2n-2$ moments can be found:
  \begin{align}
    \sum\limits_{j,s,t,k=0}^{n-1}\,\sum\limits_{l=0}^{s+t} c_m^{jk} G^{-1}_{js} c_l^{st} G^{-1}_{tk} \Braket{Q_l}_P  &= \gamma_m
    \label{linsystemP}
  \end{align}
  Then construct $\Braket{Q_jQ_k}_P$ Gram matrix of the measure
  corresponding to found moments $\Braket{Q_l}_P$,
  this gives the required $P(x)=\sum_{j,s,t,k=0}^{n-1}Q_j(x)G^{-1}_{js} \Braket{Q_sQ_t}_P G^{-1}_{tk} Q_k(x)$.
  To construct $\|\rho\|$ operator, eigenvalues/eigenvectors of which give (\ref{rhoPrelation}):
  solve (\ref{GEV}) generalized eigenvalue problem with
  the matrices $\Braket{Q_jQ_k}_P$
  and  $\Braket{Q_jQ_k}$
  in (\ref{GEV}) left-- and right-- hand side respectively, 
  obtained eigenvalues/eigenvectors
  pairs give (\ref{rhoPrelation}) expansion over the states of $\|\rho\|$ operator:
  \begin{align}
    \sum\limits_{k=0}^{n-1}
    \Braket{Q_jQ_k}_P
    \alpha^{[i]}_k
    &=
    \lambda^{[i]}
    \sum\limits_{k=0}^{n-1}
    \Braket{Q_jQ_k}
    \alpha^{[i]}_k
    \label{rhoGEV} \\
    \rho(x,y)&=
    \sum\limits_{i=0}^{n-1} \lambda^{[i]} \psi^{[i]}(x)\psi^{[i]}(y) =
    \sum\limits_{i=0}^{n-1} \Ket{\psi^{[i]}} \lambda^{[i]} \Bra{\psi^{[i]}} =\|\rho\|
    \label{rhoPol} \\
    P(x)&=\rho(x,x)
  \end{align}
\end{proof}
\begin{remark}
  The expansion of $P(x)=\sum_{j,s,t,k=0}^{n-1}Q_j(x)G^{-1}_{js}\Braket{Q_sQ_t}_P G^{-1}_{tk} Q_k(x)$
  with the matrix $\Braket{Q_jQ_k}_P$ generated by a measure is unique, the measure moments are
  (\ref{linsystemP}) linear system solution;
   without a requirement that the matrix to be generated by a measure,
   the solution is non--unique.
   Another non--uniqueness can arise from a degeneracy of  $\Braket{Q_jQ_k}_P$
   matrix, for example, take Christoffel function (\ref{christoffelfun}),
   $1/K(x)=P(x)=\sum_{j,k=0}^{n-1}Q_j(x)G^{-1}_{jk}Q_k(x)$:
   the solution 
   (\ref{linsystemP}) and the matrix $\Braket{Q_jQ_k}_P$ are unique,
   but the (\ref{rhoPrelation}) expansion is non--unique due to (\ref{rhoGEV})
   spectrum degeneracy
   (all the eigenvalues are equal to one),
   $1/K(x)=\sum_{i=0}^{n-1}\left[\phi^{[i]}(x)\right]^2$
   for an arbitrary orthogonal basis $\Ket{\phi^{[i]}}$.
\end{remark}
\begin{note}
  This prof is actually an algorithm to construct the density matrix $\|\rho\|$,
  producing a given polynomial $P(x)$. In provided implementation
  \texttt{\seqsplit{com/polytechnik/utils/BasisFunctionsMultipliable.java}}
  the method \texttt{\seqsplit{getMomentsOfMeasureProducingPolynomialInKK\_MQQM()}}, for a given $P(x)$,
  solves the linear system (\ref{linsystemP}) and
  obtains the moments $\Braket{Q_m}_P$. The method
  \texttt{\seqsplit{getDensityMatrixProducingGivenPolynomial()}}
  uses these moments to solve (\ref{rhoGEV})
  and
  to obtain the $\|\rho\|$ from (\ref{rhoPol})
  as a Lebesgue quadrature,
  the spectrum of which corresponds to a given polynomial $P(x)$ (\ref{rhoPrelation}).
\end{note}
From (\ref{rhoPrelation}) it immediately follows that
the sum of all $\|\rho\|$ eigenvectors
is equal to  $\Braket{P(x)}=   \sum_{i=0}^{n-1}\lambda^{[i]}$,
particularly for Christoffel function we have: $\Braket{1/K(x)}=\sum_{i=0}^{n-1}\lambda^{[i]}=n$,
and in general case:
\begin{align}
  \Braket{f(x)P(x)}&=\sum_{i=0}^{n-1}\lambda^{[i]} \Braket{\psi^{[i]}|f|\psi^{[i]}}
 \label{intrho}
\end{align}
The (\ref{intrho}) is a representation of $\Braket{f(x)P(x)}$
integral as a sum of $f$--moments over the states of the density
matrix $\|\rho\|$ operator (\ref{rhoGEV}). This formula is a complementary one
to (\ref{fP}),
which is a representation of $\Braket{f(x)P(x)}$
integral as a sum of $P$--moments over the states of $\|f\|$ operator (\ref{GEV}).

Finally, we want to emphasize,
that used all of the above $\Braket{\psi}^2$  is a special
case of a density matrix. Consider $\|\rho\|=\Ket{1}\Bra{1}$, then
$\Braket{\psi}^2=\Braket{\psi|\rho|\psi}$,
and for an operator $\|f\|$, $\Braket{f}=\mathrm{Spur}\, \|f|\rho\|$
Similarly, a spur with a density matrix $\|\rho\|$,
e.g. corresponding to a polynomial $P(x)$,
can be used instead of all averages:
\begin{align}
  \Braket{f}&\to \mathrm{Spur}\, \|f|\rho\|
  \label{AverageReplacement}
\end{align}
This way the  approach we developed can be extended not only to polynomial by operator products study, but also
to operator--by--operator products. Then, instead of
$\mathrm{Spur}\, \|f|\rho\|$, which can be written either in (\ref{fP}) or in (\ref{intrho})
representation,
a general case of two operators $\mathrm{Spur}\, \|f|g\|$
can be considered. The first attempt to explore this direction is presented in 
\cite{ArxivMalyshkinJointDistribution}.

\section{\label{christoffelSpectrum}On The Christoffel Function Spectrum}
In the consideration above $f$ was  a given function
with finite moments $\Braket{Q_j | f |Q_k }$ in (\ref{QfQmoments}).
It's selection depends on the problem approached, for example
we used $f=x$ to obtain Gaussian quadrature (\ref{GEVgauss})
and $f=dC/dN$ for
Li--ion degradation rate study
in Fig. \ref{figLiIon}.
A question arise what the result we can expect if
the Christoffel function (\ref{christoffelfun})
is used as
$f(x)=K(x)=1\Big/\sum_{j,k=0}^{n-1}Q_j(x)G^{-1}_{jk}Q_k(x)$.
\begin{theorem}
  \label{theoremFeqChristoffel}
  If $f(x)$ is equal to the Christoffel function $K(x)$
  the eigenproblem
  \begin{align}
    \sum\limits_{k=0}^{n-1} \Braket{Q_j|K(x)|Q_k} \alpha^{[i]}_k &=
    \lambda^{[i]}_K \sum\limits_{k=0}^{n-1} \Braket{ Q_j|Q_k} \alpha^{[i]}_k
    \label{GEVchristoffelQQ} \\
    \psi^{[i]}_K(x)&=\sum\limits_{k=0}^{n-1} \alpha^{[i]}_k Q_k(x)
  \label{psiCchristoffel}
  \end{align}
  has the sum of all eigenvalues $\lambda^{[i]}_K$ 
   equals to the total measure:
\begin{align}
\Braket{1}&=\int d\mu=\sum_{i=0}^{n-1}\lambda^{[i]}_K
\label{Spurrhok}
\end{align}
\end{theorem}
\begin{proof}
  For a given $n$ Christoffel function $K(x)$ vanishes
  at large $x$ with $1/x^{2n-2}$ asymptotic,
  the integrals (\ref{QfQmoments}) are finite and (\ref{GEVchristoffelQQ})
  has a solution with eigenvalues $\lambda^{[i]}_K$ (possibly degenerated)
  and eigenfunctions $\psi^{[i]}_K(x)$.
  The Christoffel function (\ref{christoffelfun}) can be expressed
  in any orthogonal basis, take $\phi^{[i]}(x)=\psi^{[i]}_K(x)$.
  From $\lambda^{[i]}_K=\Braket{\psi^{[i]}_K|K(x)|\psi^{[i]}_K}=\Braket{\left[\psi^{[i]}_K(x)\right]^2K(x)}$
  and $K(x)=1\Big/\sum_{i=0}^{n-1} \left[\psi^{[i]}_K(x)\right]^2$
obtain $\Braket{1}=\sum_{i=0}^{n-1}\lambda^{[i]}_K$.
\end{proof}
The eigenfunctions (\ref{GEVgauss})  of a Gaussian quadrature
correspond to $x$--localized states, they are $\|x\|$ operator eigenfunctions
and the total weight is $\Braket{1}=\sum_{i=0}^{n-1}K(x^{[i]})$
with $w^{[i]}=K(x^{[i]})=\Braket{\psi^{[i]}}^2$; the $\psi^{[i]}(x)$ is
(\ref{GEVgauss}) eigenproblem solution.
The states $\psi^{[i]}_K(x)$  of (\ref{GEVchristoffelQQ}) eigenproblem
satisfy
 Theorem \ref{theoremFeqChristoffel}
and the Lebesgue quadrature weights sum (\ref{Lebweightssum}):
$\Braket{1}=\sum_{i=0}^{n-1}\Braket{\psi^{[i]}_K|K(x)|\psi^{[i]}_K}=
\sum_{i=0}^{n-1}\Braket{\psi^{[i]}_K}^2$.
However
an eigenvalue $\lambda^{[i]}_K$ of (\ref{GEVchristoffelQQ})
\textbf{is not equal} to the Lebesgue quadrature weight $\Braket{\psi^{[i]}_K(x)}^2$,
see (\ref{dMu2}) below.
A density matrix operator
can be constructed from (\ref{GEVchristoffelQQ}) eigenvalues and eigenfunctions:
\begin{align}
\rho_K(x,y)&=
\sum\limits_{i=0}^{n-1} \lambda^{[i]}_K \psi^{[i]}_K(x)\psi^{[i]}_K(y) =
    \sum\limits_{i=0}^{n-1} \Ket{\psi^{[i]}_K} \lambda^{[i]}_K \Bra{\psi^{[i]}_K} =\|\rho_K\|
\label{rhoChristoffel}
\end{align}
it is similar to ``regular average'' density matrix $\|\rho\|=\Ket{1}\Bra{1}$
considered in the Appendix \ref{densitymatrix}, e.g. both have
the same Spur (equals to total measure).
The (\ref{rhoChristoffel}) is the same as (\ref{rhoPol})
but the  eigenvalues/eigenfunctions are (\ref{GEVchristoffelQQ})
instead of (\ref{rhoGEV}).
The density matrix operator $\|\rho_K\|$
 corresponds to 
the Christoffel function $K(x)$.
The problem of averaging an operator $\|g\|$
with the Christoffel function used as a weight
is a difficult problem \cite{2015arXiv151107085G}.
The (\ref{rhoChristoffel}) allows this problem
to be approached directly: take the $\mathrm{Spur}\, \|g|\rho_K\|$.
A question arise about $\|\rho_K\|\Leftrightarrow K(x)$ mapping: whether
it is a one--to--one mapping or not?
For $1/K(x)$, a polynomial of $2n-2$ degree,
the mapping is (\ref{rhoPol}).
For $K(x)$ this requires a separate consideration.
Anyway, built from the Christoffel function
density matrix operator (\ref{rhoChristoffel})
allows us to consider an operator average with the
Christoffel function in a regular
``operatorish'' way: by taking a Spur of operators product.

Recent progress\cite{malyshkin2019radonnikodym}
in numerical computability
of Radon--Nikodym derivative
for multi--dimensional $\mathbf{x}$
allows us to demonstrate Theorem \ref{theoremFeqChristoffel} numerically.
Take a simple $d\mu=dx$ demonstration
measure of the  Appendix C of \cite{malyshkin2019radonnikodym}:
\begin{align}
  d\mu&=dx \label{totalmeasurerungeF} \\
  x&\in [-1:1] \nonumber
\end{align}
The file \texttt{\seqsplit{dataexamples/runge\_function.csv}}
is bundled with
\href{http://www.ioffe.ru/LNEPS/malyshkin/code_polynomials_quadratures.zip}{provided software}.
It has 10001 rows (the measure support is split to 10000 intervals)
and 9 columns.
In the first seven columns there are the
powers of $x$: $1,x,x^2,x^3,x^4,x^5,x^6$. Then, in the next two columns,
follow:
\href{https://en.wikipedia.org/wiki/Runge%27s_phenomenon}{Runge function}
$1/(1+25x^2)$ and the (\ref{totalmeasurerungeF}) weight.  
Run the program to obtain Christoffel function value for all observations
in data file   
(column indexes are base 0):
\begin{verbatim}
java com/polytechnik/utils/RN --data_cols=9:0,6:1:8:1 \
      --data_file_to_build_model_from=dataexamples/runge_function.csv
\end{verbatim}
Here as $f$ we use the $x$, the data is in the column with index $1$.
The Lebesgue quadrature then produces
the Gaussian quadrature for the measure (\ref{totalmeasurerungeF}):
\begin{equation}
  \begin{aligned}
x^{[0]}&= -0.9491080257215085 & w^{[0]}&= 0.1294848235792277 & w_K^{[0]}&= 0.11746154871932572 \\
x^{[1]}&= -0.7415313130354606 & w^{[1]}&= 0.279705429437816 & w_K^{[1]}&= 0.2794795769155739 \\
x^{[2]}&= -0.40584522389537203 & w^{[2]}&= 0.3818301175303132 & w_K^{[2]}&= 0.38911964330481996 \\
x^{[3]}&= 0 & w^{[3]}&= 0.41795925890484187 & w_K^{[3]}&= 0.42787846212051234 \\
x^{[4]}&= 0.405845223895157 & w^{[4]}&= 0.3818301175306451 & w_K^{[4]}&= 0.38911964330486587 \\
x^{[5]}&= 0.7415313130353846 & w^{[5]}&= 0.2797054294378024 & w_K^{[5]}&= 0.27947957691558917 \\
x^{[6]}&= 0.9491080257213823 & w^{[6]}&= 0.12948482357916594 & w_K^{[6]}&= 0.11746154871930853
\end{aligned}
  \label{dMu2Gauss}
\end{equation}
A small difference between (\ref{dMu2Gauss}) and exact values
of 7-point Gaussian quadrature for the measure (\ref{totalmeasurerungeF})
is due to the fact that
the moments calculation is not exact, they are calculated
from 10001 discrete points
in the file \texttt{\seqsplit{dataexamples/runge\_function.csv}}.
The Christoffel weights $w_K^{[i]}$ (\ref{ChristoffewlWeights}) are
close to $w^{[i]}$ in case $f=x$.
Created file \texttt{\seqsplit{runge\_function.csv.RN.csv}}
has 22 columns.
First column is the label,
next 7 columns are the powers of $x$ (copied from input), then $f=x$, weight,
Radon--Nikodym derivative (\ref{bnevai}) of $fd\mu$ and $d\mu$ (here $f=x$),
and the Christoffel function $K(x)$ (\ref{christoffelfun})
is in the column with index $12$; the other columns follow to total 22.
Run the program again using the
Christoffel function as $f$ (Christoffel function
is in the column with index 12; an alternative is to use \texttt{\seqsplit{--flag\_replace\_f\_by\_christoffel\_function=true}}):
\begin{verbatim}
java com/polytechnik/utils/RN --data_cols=22:1,7:12:9:0 \
      --data_file_to_build_model_from=runge_function.csv.RN.csv
or
java com/polytechnik/utils/RN --data_cols=9:0,6:1:8:1 \
      --flag_replace_f_by_christoffel_function=true \
      --data_file_to_build_model_from=dataexamples/runge_function.csv
\end{verbatim}
The output file
\texttt{\seqsplit{runge\_function.csv.RN.csv.RN.csv}}
now contains the eigenvalues $\lambda^{[i]}_K$ and
the Lebesgue weights $w^{[i]}$ for eigenproblem (\ref{GEVchristoffelQQ})
with the measure (\ref{totalmeasurerungeF}):
\begin{equation}
  \begin{aligned}
\lambda_K^{[0]}&= 0.10226835684407387 & w^{[0]}&= 0.16153573777120298 & w_K^{[0]}&= 0.10226835684403417 \\
\lambda_K^{[1]}&= 0.12057295282629424 & w^{[1]}&= 0 & w_K^{[1]}&= 0.12057295282626747 \\
\lambda_K^{[2]}&= 0.25910242661821975 & w^{[2]}&= 0.4476418241676696 & w_K^{[2]}&= 0.2591024266180915 \\
\lambda_K^{[3]}&= 0.2924778951810179 & w^{[3]}&= 0 & w_K^{[3]}&= 0.2924778951809419 \\
\lambda_K^{[4]}&= 0.37696956667653253 & w^{[4]}&= 0.6388507741017023 & w_K^{[4]}&= 0.37696956667633214 \\
\lambda_K^{[5]}&= 0.4079988698735509 & w^{[5]}&= 0 & w_K^{[5]}&= 0.40799886987353334 \\
\lambda_K^{[6]}&= 0.44060993198085746 & w^{[6]}&= 0.751971663959237 & w_K^{[6]}&= 0.44060993198079923
  \end{aligned}
  \label{dMu2}
\end{equation}
We see that for $f(x)=K(x)$ both: the eigenvalues sum and the Lebesgue quadrature weights sum
are equal to total measure, it is $2$ for (\ref{totalmeasurerungeF}).
Some of the Lebesgue quadrature weights are equal to $0$;
for (\ref{totalmeasurerungeF}) measure
Christoffel function is even, there are even and odd eigenfunctions,
the average of odd eigenfunctions is zero.
All Christoffel weights $w_K^{[i]}$ from (\ref{ChristoffewlWeights}) are non--zero
and coincide with $\lambda_K^{[i]}$ because $f(x)=K(x)$,
they will not coincide if optimal clustering to $D<n$ is performed with $\|\rho\|=\Ket{1}\Bra{1}$, see Appendix \ref{optimalClustering} below.

For  a given  $f(x)$ an eigenfunction $\psi^{[i]}(x)$
of eigenproblem (\ref{GEV}) 
may possibly produce zero weight in
the Lebesgue quadrature, this can be an inconvenient feature
in a practical situation.
The operator $\|\rho_K\|$  (\ref{rhoChristoffel}) allows us 
to introduce the ``Christoffel weights'' $w_K^{[i]}$, that are always positive.
The operator $\|\rho_K\|$  Spur  (\ref{Spurrhok}) is calculated in $\Ket{\psi^{[i]}_K}$ basis,
it is equal to total measure $\Braket{1}$. The Spur is invariant
with respect to basis transform, it  will be the same when written in
$\Ket{\psi^{[i]}}$ basis, (\ref{GEV}) eigenvectors.
\begin{align}
\Braket{1}&=\sum_{i=0}^{n-1} \Braket{\psi^{[i]}_K|\rho_K|\psi^{[i]}_K}=
\sum_{i=0}^{n-1} \Braket{\psi^{[i]}|\rho_K|\psi^{[i]}}
\label{SpurrhokBasis2}
\end{align}
Define ``Christoffel weights'' $w_K^{[i]}$
as an alternative to the ``Lebesgue weights'' $w^{[i]}=\Braket{\psi^{[i]}}^2$ (\ref{wiLeb})
\begin{align}
w_K^{[i]}&=\Braket{\psi^{[i]}|\rho_K|\psi^{[i]}}=\Braket{\psi^{[i]}|K(x)|\psi^{[i]}}
=\Braket{\frac{\left[\psi^{[i]}(x)\right]^2}{
\sum_{j=0}^{n-1}\left[\psi^{[j]}(x)\right]^2
}}
\label{ChristoffewlWeights}
\end{align}
The weights $w_K^{[i]}$ satisfy the same normalizing condition (\ref{SpurrhokBasis2})
as the Lebesgue weights normalizing (\ref{Lebweightssum}). In Fig. \ref{figLiIonWk}
the Christoffel weights are compared to (\ref{wiLeb}) weights.
One can see these weights are very close. However, the Christoffel weights $w_K^{[i]}$
have a property of being always positive and are related to Christoffel function operator $\|\rho_K\|$.
\begin{figure}[t]
\includegraphics[width=5cm]{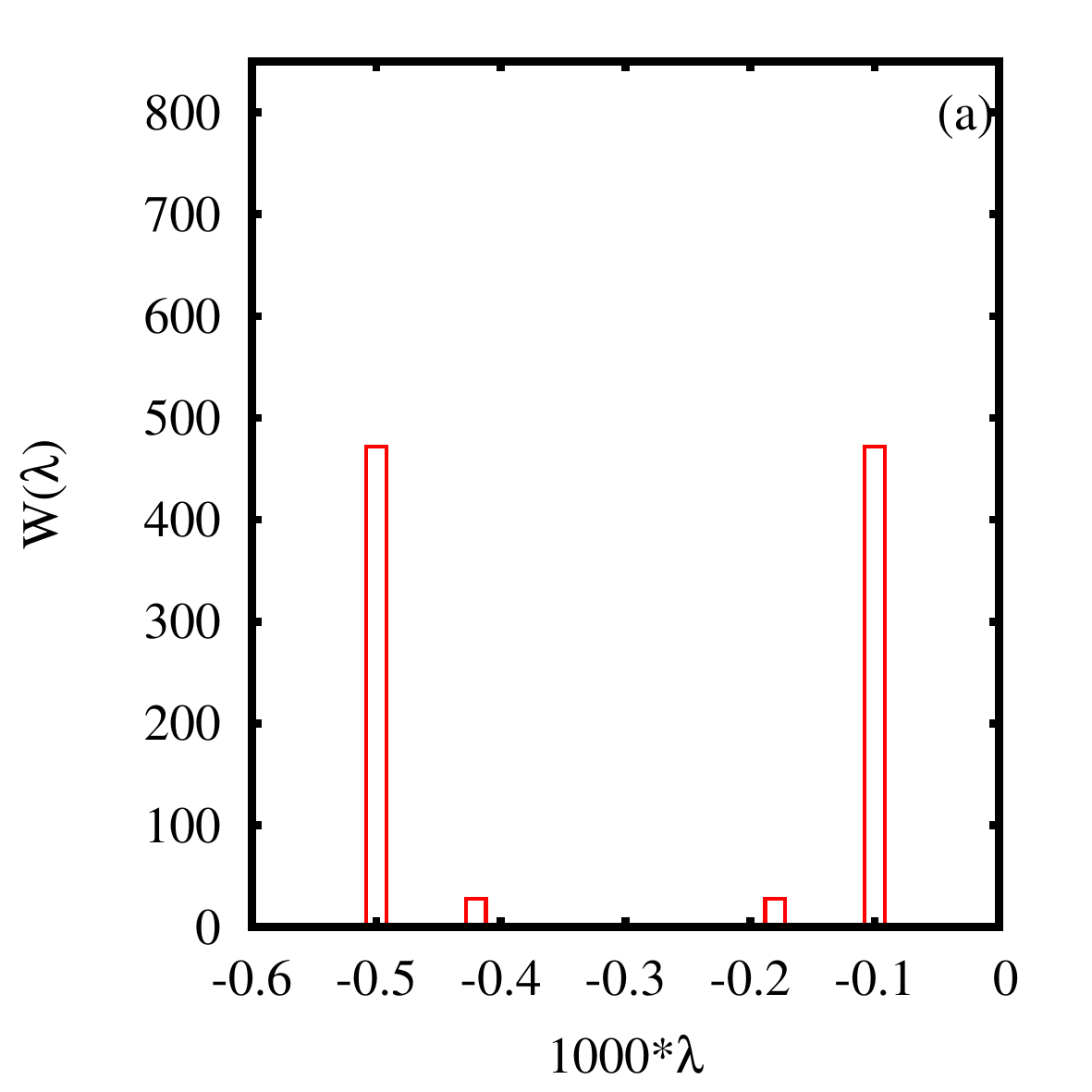}
%
\includegraphics[width=5cm]{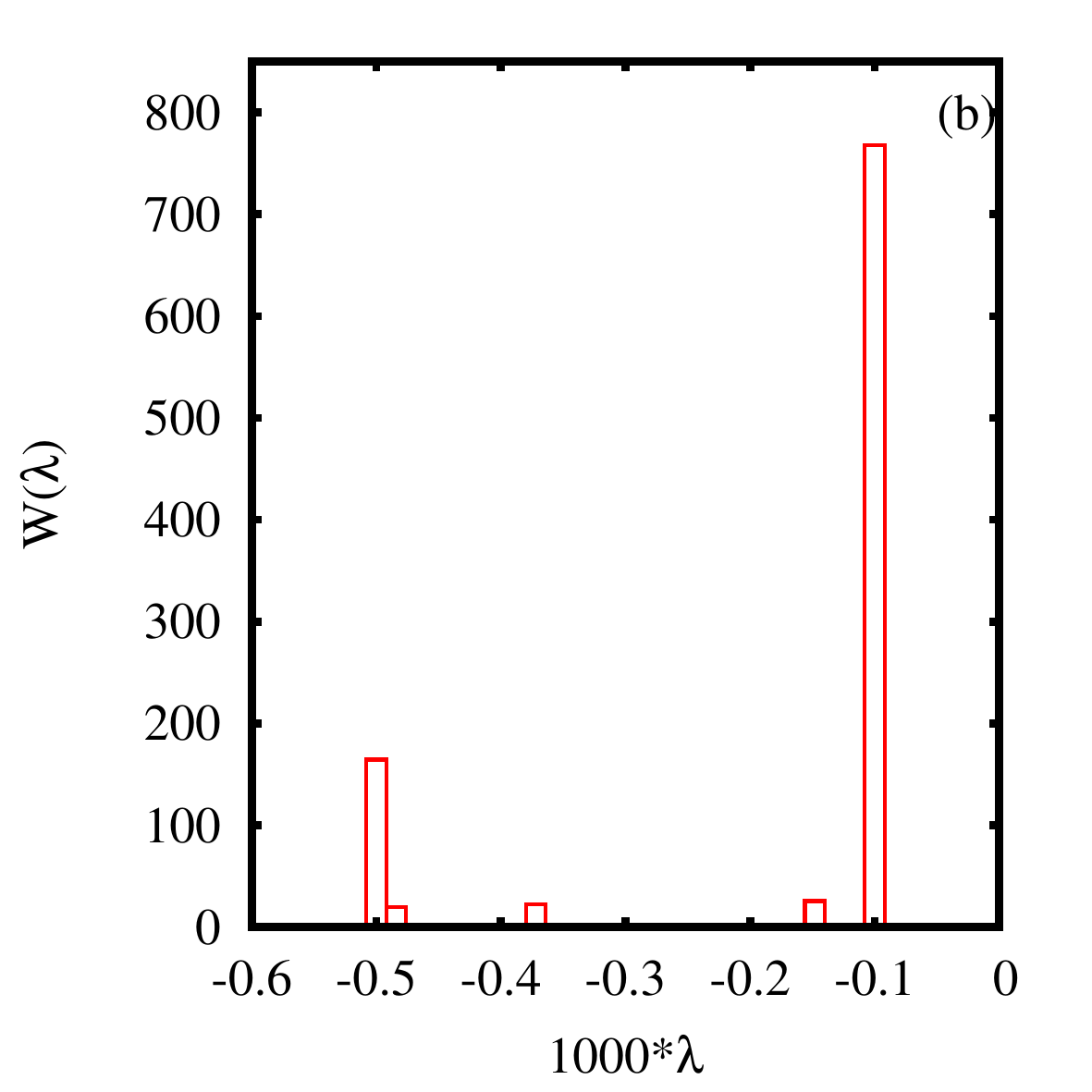}
 \caption{\label{figLiIonWk}
  The same model as in Fig. \ref{figLiIon},
  but with the Christoffel weights (\ref{ChristoffewlWeights})
  instead of (\ref{wiLeb}) weights; (a) corresponds to
\hyperref[figLiIon]{\ref*{figLiIon}c} and (b)
  corresponds to \hyperref[figLiIon]{\ref*{figLiIon}f}.
  As in Fig. \ref{figLiIon}
   the peak height corresponds exactly to the stage length because of chosen measure $d\mu=dN$.
   The calculations are performed for  $n=50$ in polynomial basis.   
  }
\end{figure}

The eigenvalues of (\ref{GEV})
are the Lebesgue integral (\ref{intLebesgue}) value--nodes $f^{[i]}$,
the weights are obtained from eigenfunction $\Ket{\psi^{[i]}}$ average.
As we emphasized above in (\ref{AverageReplacement}),
any average corresponds to some density matrix.
The $\|\rho\|=\Ket{1}\Bra{1}$ corresponds to  a ``regular'' average,
the Lebesgue weights then are: $w^{[i]}=\Braket{\psi^{[i]}|\rho|\psi^{[i]}}$.
The $\|\rho_K\|$ corresponds to ``Christoffel function average''
with the weights (\ref{ChristoffewlWeights}).

The calculation of ``Christoffel weights'' requires
\textbf{one more matrix} $\Braket{Q_j|K(x)|Q_k}$ to be calculated
from the data sample.
The cost to pay for the ``Christoffel weights'' is that the data sample now
should be processed twice:
\begin{itemize}
\item Construct $\Braket{Q_j|Q_k}$ and $\Braket{Q_j|f|Q_k}$.
\item
For every observation calculate Christoffel function $K(x)$ from the matrix $\Braket{Q_j|Q_k}$.
Build the matrix $\Braket{Q_j|K(x)|Q_k}$.
\end{itemize}
A second pass is required because Christoffel function
matrix elements $\Braket{Q_j|K(x)|Q_k}$ go beyond basis function products and should be
evaluated directly.
In addition to the matrix of outcomes  $\Braket{Q_j|f|Q_k}$
we now have a matrix of ``coverage'' $\Braket{Q_j|K(x)|Q_k}$
which is used to obtain
operator $\|\rho_K\|$,
corresponding to the Christoffel function $K(x)$.
The Christoffel function can be considered as
a ``proxy'' for coverage\cite{2015arXiv151107085G,lasserre2019empirical,beckermann2018perturbations}:
the number of observations that are ``close enough'' to a given $x$;
but it can estimate only the coverage of a ``localized'' at $x$ state,
not the coverage of a given state $\Ket{\psi}$.
In contradistinction to the Christoffel function $K(x)$,
the Christoffel function density matrix $\|\rho_K\|$ (\ref{rhoChristoffel})
can estimate the coverage of any given state $\Ket{\psi}$
as $\Braket{\psi|\rho_K|\psi}$; it is
not limited to localized states as the Christoffel function $K(x)$ is.

A uniqueness of the Lebesgue quadrature
makes it a very attractive tool for data analysis.
When a data analysis problem defines some $f$,
for example  Li--ion degradation rate
$f=dC/dN$ in Fig. \ref{figLiIon},
a class label in ML \cite{malyshkin2019radonnikodym},
gray intensity in image reconstruction\cite{2015arXiv151101887G}, etc.
the solution
$(\lambda^{[i]},\psi^{[i]})$
of
(\ref{GEV})
is unique and can be used as a basis for: PCA expansion
(\ref{evexpansionStdev}), $f$ distribution
estimation (\ref{wiLeb}) or (\ref{ChristoffewlWeights}),
optimal clustering of Appendix \ref{optimalClustering}, etc.
There is a setup where a function $f$
either cannot be defined or is a
\href{https://en.wikipedia.org/wiki/Multivalued_function}{multivalued function}
for which an eigenvalue problem cannot be formulated.
However, we still want to obtain a \textbf{unique} basis
that is constructed from the data sample,
for example to avoid
\href{https://en.wikipedia.org/wiki/Principal_component_analysis#Properties_and_limitations_of_PCA}{PCA dependence on attributes scale}.
In this case the most straightforward approach
is to take the Christoffel function as $f(x)=K(x)$.
This approach can be easily extended to a multi--dimensional $\mathbf{x}$,
see \cite{malyshkin2019radonnikodym}. An issue that often arise
in case of a multi--dimensional $\mathbf{x}$ is a degeneracy
of Gram matrix $G_{jk}=\Braket{Q_jQ_k}$. In the Appendix A
of \cite{malyshkin2019radonnikodym} a regularization algorithm
is presented, it needs to be applied to $\mathbf{x}$ to obtain
a regularized basis $\mathbf{X}$.
Then, in the regularized basis,
the Christoffel function (\ref{christoffelfun}) can be
calculated\footnote{
See the method \texttt{\seqsplit{com/polytechnik/utils/DataRegularized.java:getRNatXoriginal(double [] xorig).getChristoffelOatX()}}
of \href{http://www.ioffe.ru/LNEPS/malyshkin/code_polynomials_quadratures.zip}{provided software}
calculating the $1/K(\mathbf{x})$.},
the eigenproblem (\ref{GEVchristoffelQQ})
solved, and a unique basis $\psi^{[i]}_K(x)$ obtained!

\section{\label{optimalClustering}On The Optimal Clustering Problem With A Density Matrix Average}
The most noticeable result of our work
\cite{malyshkin2019radonnikodym}
is basis reduction algorithm, Section ``Optimal Clustering''.
For $n$ input attributes (such as $Q_k(x)$ or multi--dimensional $\mathbf{x}$)
construct $D\le n$ linear combinations of them $\psi_G^{[m]}(x)$, $m=0\dots D-1$,
that optimally
separate $f$ in terms of $\Braket{f\psi^2}/\Braket{\psi^2}$.
This solution is the key concept of our approach to data overfitting problem.
A sketch of \cite{malyshkin2019radonnikodym} theory:
\begin{itemize}
\item
Solve (\ref{GEV}), obtain $n$ pairs $(f^{[i]}=\lambda^{[i]},\psi^{[i]})$.
Introduce a measure $\Braket{\cdot}_L$
\begin{align}
  \Braket{g(f)}_L&=\sum_{i=0}^{n-1} g(f^{[i]})w^{[i]} \label{Lmeasure} \\
   w^{[i]}&=\Braket{\psi^{[i]}}^2 \label{LmeasureWeight}
\end{align}
\item Construct a $D$--point Gaussian quadrature in $f$--space with the measure $\Braket{\cdot}_L$,
obtain the functions
$\psi_G^{[m]}(f)$ in $f$--space (Eq. (\ref{GEVgauss}) of dimension $D$
with $f$ used instead of $x$).
The optimization problem in $f$--space is solved only once,
all the solutions in $x$--space are obtained from the $\psi_G^{[m]}(f)$.
This is different from \cite{marx2019tractable}
where for every given $x$ a conditional minimization
of the polynomial $1/K(\widetilde{x})$ is required:
for a fixed $x$ in $\widetilde{x}=(x,f)$ find the $f$ providing the minimum.
\item
Convert the optimal clustering  solution $\psi_G^{[m]}(f)$
from $f$--space to $x$--space, obtain $\psi_G^{[m]}(x)$. This
conversion is possible only because the Lebesgue weights (\ref{LmeasureWeight})
are used in (\ref{Lmeasure}).
\end{itemize}

The Lebesgue weights $w^{[i]}=\Braket{\psi^{[i]}}^2$ correspond
to a very specific form of the density matrix
$\|\rho\|=\Ket{1}\Bra{1}$ (a ``regular'' average), this density
matrix operator is a
\href{https://en.wikipedia.org/wiki/Quantum_state#Pure_states}{pure state}.
A question arise whether the optimal clustering success of Ref. \cite{malyshkin2019radonnikodym}
can be repeated with a more general form of the density matrix,
e.g. with the $\|\rho_K\|$ from (\ref{rhoChristoffel})?
Introduce a measure $\Braket{\cdot}_L$
\begin{align}
  \Braket{g(f)}_L&=\sum_{i=0}^{n-1} g(f^{[i]})  w^{[i]} \label{LmeasureDM} \\
  w^{[i]}&=\Braket{\psi^{[i]}|\rho|\psi^{[i]}} \label{LmeasureWeightDM}
\end{align}
The weights (\ref{LmeasureWeightDM}) is 
the most general form  of the Lebesgue weighs;
(\ref{wiLeb}) corresponds to $\|\rho\|=\Ket{1}\Bra{1}$.

As in \cite{malyshkin2019radonnikodym} a $D$--point Gaussian quadrature
can be constructed from (\ref{LmeasureDM}) measure,
the eigenfunctions $\psi_G^{[m]}(f)$ are (\ref{GEVgauss}) eigenvectors
with the
replace: $n\rightarrow D$ and $x\rightarrow f$.
They are orthogonal as 
\begin{subequations}
\label{psiGfOrt}
\begin{align}
\delta_{ms}&=\Braket{\psi_G^{[m]}(f)|\psi_G^{[s]}(f)}_L \\
\lambda_G^{[m]}\delta_{ms}&=\Braket{\psi_G^{[m]}(f)|f|\psi_G^{[s]}(f)}_L \\
w_G^{[m]}&=
\Braket{\psi_G^{[m]}}_L^2
=\frac{1}{\left[\psi_G^{[m]}(\lambda_G^{[m]})\right]^2}
\end{align}
\end{subequations}
The problem is to convert obtained optimal clustering solution
$\psi_G^{[m]}(f)$
from $f$ to $x$ space; $D$ eigenvalues are denoted  as $\lambda_G^{[m]}$
in order to not to mistake them with $n$ eigenvalues $f^{[i]}$ of (\ref{GEV}).
Introduce $D$ operators $\|\Psi_G^{[m]}\|$, ($m=0\dots D-1$; $i=0\dots n-1$):
\begin{align}
\|\Psi_G^{[m]}\|&=\sum\limits_{i=0}^{n-1}\Ket{\psi^{[i]}}
\psi_G^{[m]}(f^{[i]})\Bra{\psi^{[i]}} \label{psiDM} \\
\Braket{A}_{\rho}&=
\mathrm{Spur}\, \|A|\rho\| \label{psiDMaver}
\end{align}
In the basis of (\ref{GEV}) eigenproblem the operators
$\|\Psi_G^{[m]}\|$ are diagonal.
With (\ref{psiDMaver}) definition of average  the orthogonality relation for $\|\Psi_G^{[m]}\|$
with respect to $\Braket{\cdot}_{\rho}$ is the same as  (\ref{psiGfOrt})
for $\psi_G^{[m]}(f)$ with respect to the measure $\Braket{\cdot}_L$:
\begin{subequations}
\label{psiGOperOrt}
\begin{align}
\delta_{ms}&=\Braket{\Psi_G^{[m]}|\Psi_G^{[s]}}_{\rho} \\
\lambda_G^{[m]}\delta_{ms}&=\Braket{\Psi_G^{[m]}|f|\Psi_G^{[s]}}_{\rho} \\
w_G^{[m]}&=\Braket{\Psi_G^{[m]}}_{\rho}^2
\end{align}
\end{subequations}
For $\|\rho\|=\Ket{1}\Bra{1}$
the $\psi_G^{[m]}(x)$
of \cite{malyshkin2019radonnikodym} can be expressed via
the operators $\|\Psi^{[m]}_G\|$
\begin{align}
\Ket{\psi_G^{[m]}}&=\Ket{\Psi^{[m]}_G\Big|1} \label{psiHxConversion} \\
p^{[m]}(x)&=\left[\psi_G^{[m]}(x)\right]^2 \label{pm} \\
f_{RN}(x)&=
\frac{\sum\limits_{m=0}^{D-1} \lambda_G^{[m]} p^{[m]}(x)}
{\sum\limits_{m=0}^{D-1} p^{[m]}(x)} \label{fRN} \\
f_{RNW}(x)&=
\frac{\sum\limits_{m=0}^{D-1} \lambda_G^{[m]} p^{[m]}(x) w_G^{[m]}}
{\sum\limits_{m=0}^{D-1} p^{[m]}(x) w_G^{[m]}} \label{fRNW}
\end{align}
The optimal clustering states $\psi_G^{[m]}(f)$ can only be
converted to pure states in $x$--space $\psi_G^{[m]}(x)$ 
when the density matrix $\|\rho\|$ is of a pure state form $\Ket{\varphi}\Bra{\varphi}$,
otherwise the conversion to $x$--space produces mixed states
described by the operators $\|\Psi_G^{[m]}\|$.
While the $\psi_G^{[m]}(x)$ does not exist for a general $\|\rho\|$,
the $p^{[m]}(x)$ weight, required
to obtain Radon--Nikodym interpolation (\ref{fRN}) and
classification (\ref{fRNW}) solutions,
can always be obtained.
From (\ref{psiDM}) it follows that
\begin{align}
p^{[m]}(x)&=
\Braket{\psi_x\Big|\Psi_G^{[m]}|\rho|\Psi_G^{[m]}|\psi_x}=
\sum\limits_{i,j=0}^n
\psi^{[i]}(x)\psi^{[j]}(x)
\psi_G^{[m]}(f^{[i]})\psi_G^{[m]}(f^{[j]})\Braket{\psi^{[i]}|\rho|\psi^{[j]}}
\label{pmDM}
\end{align}
For $\|\rho\|=\Ket{1}\Bra{1}$ (\ref{pmDM}) becomes (\ref{pm}).
A very important feature of the Radon--Nikodym approach (\ref{fRN})
is that it can be generalized to the density matrix states.
The $\left[\psi_G^{[m]}(x)\right]^2$ used as
an eigenvalue weight needs to be replaced
by a more general form (\ref{pmDM}). Thus
all the optimal clustering
results of Ref. \cite{malyshkin2019radonnikodym}
are now generalized from the weights (\ref{wiLeb}) to the weights
$\Braket{\psi^{[i]}|\rho|\psi^{[i]}}$, described by a density matrix
$\|\rho\|$ of the most general form, e.g. by the
Christoffel function density matrix (\ref{rhoChristoffel}).

\section{\label{RNFDFusage}Usage Example of \texttt{{com/polytechnik/algorithms/ExampleRadonNikodym\_F\_and\_DF.java}}}
The \texttt{\seqsplit{com/polytechnik/algorithms/ExampleRadonNikodym\_F\_and\_DF.java}}
is a program processing 1D data.
It was used in \cite{2016arXiv161107386V} to obtain  relaxation rate distribution. In contrast with advanced multi--dimensional
approach of \cite{malyshkin2019radonnikodym},
this program has a rigid interface and limited functionality.
It is bundled with \href{http://www.ioffe.ru/LNEPS/malyshkin/code_polynomials_quadratures.zip}{provided software}.
Usage example to reproduce Fig. \ref{figLiIon} data:
Create a two--stage linear model of Fig.
\hyperref[figLiIon]{\ref*{figLiIon}d} with 800:200 lengths,
save the model to  \texttt{\seqsplit{slope\_800\_200.csv}}.
\begin{verbatim}
java com/polytechnik/algorithms/PrintFunTwoLinearStages \
     slope_800_200.csv 10000 1000 800 1e-4 5e-4 0
\end{verbatim}
Solve (\ref{GEV}) for $f=dC/dN$ ($f=C$ is also calculated).
Use $n=50$ and the data from \texttt{\seqsplit{slope\_800\_200.csv}}.
\begin{verbatim}
java com/polytechnik/algorithms/ExampleRadonNikodym_F_and_DF \
     slope_800_200.csv 50 sampleDX
\end{verbatim}
The files  \texttt{\seqsplit{slope\_800\_200.csv.QQdf\_QQ\_spectrum.dat}}
and \texttt{\seqsplit{slope\_800\_200.csv.QQdf\_QQ\_spectrum.dat}}
are generated. They correspond to $f=dC/dN$ and to $f=C$ respectively.
The files contain 5 columns: eigenvalue index, eigenvalue $\lambda^{[i]}$,
$x_{\psi^{[i]}}=\Braket{\psi^{[i]}|x|\psi^{[i]}}/\Braket{\psi^{[i]}|\psi^{[i]}}$,
$w^{[i]}$ weight (\ref{wiLeb}),
and $w_K^{[i]}$ weight (\ref{ChristoffewlWeights}).
The data can be grouped to 25 bins of $\lambda^{[i]}$
(the column with index 1)
to produce Fig.
\hyperref[figLiIon]{\ref*{figLiIon}f}
(the weight is in the column with index 3)
and Fig. \hyperref[figLiIonWk]{\ref*{figLiIonWk}b}
(the weight is in the column with index 4).
\begin{verbatim}
java com/polytechnik/algorithms/HistogramDistribution \
     slope_800_200.csv.QQdf_QQ_spectrum.dat 5:1:3 25 >W_800_200.csv
java com/polytechnik/algorithms/HistogramDistribution \
     slope_800_200.csv.QQdf_QQ_spectrum.dat 5:1:4 25 >WK_800_200.csv
\end{verbatim}

\bibliography{LD}

\end{document}